\documentclass[11pt,reqno,tbtags]{amsart}

\usepackage{amssymb,amsfonts,amsmath,amsthm,amscd,dsfont,mathrsfs}
\usepackage{url}

\numberwithin{equation}{section}

\newcommand{\sig}{\sigma}
\newcommand{\ra}{\rightarrow}
\newcommand{\ah}{\alpha}
\newcommand{\ben}{\begin{enumerate}}
\newcommand{\een}{\end{enumerate}}
\newcommand{\beq}{\begin{equation}}
\newcommand{\eeq}{\end{equation}}
\newcommand{\bs}{\begin{slide}}
\newcommand{\es}{\end{slide}}
\newcommand{\bc}{\begin{center}}
\newcommand{\ec}{\end{center}}

\newcommand{\lam}{\lambda}
\newcommand{\Lam}{\Lambda}

\newcommand{\Del}{\Delta}
\newcommand{\Gam}{\Gamma}
\newcommand{\gam}{\gamma}

\newtheorem{theorem}{Theorem}[section]

\newtheorem{lemma}[theorem]{Lemma}
\newtheorem{corollary}[theorem]{Corollary}

\theoremstyle{definition}

\newtheorem{define}{Definition}
\newtheorem{remark}[theorem]{Remark}

\def\E{\mathbb E}

\def\reals{\mathbb{R}}

\def\P{\Pr} 

\newcommand\COND{\ensuremath{\mathbf{COND}}}
\newcommand\CONDLG{\ensuremath{\mathbf{COND}(\Lambda,\Gamma)}}
\newcommand\xCOND{\mid\COND}
\newcommand\Var{{\operatorname{Var}}}

\newcommand\ii{\mathrm{i}}

\newcommand{\sumin}{\sum_{i=1}^n}
\newcommand\tL{{\widetilde L}}
\newcommand\hL{\widehat L}
\newcommand\dd{\,\mathrm{d}}
\newcommand\floor[1]{\lfloor#1\rfloor}
\newcommand\bigpar[1]{\bigl(#1\bigr)}
\newcommand\Bigpar[1]{\Bigl(#1\Bigr)}
\newcommand\biggpar[1]{\biggl(#1\biggr)}
\newcommand\lrpar[1]{\left(#1\right)}
\newcommand\bigsqpar[1]{\bigl[#1\bigr]}
\newcommand\Bigsqpar[1]{\Bigl[#1\Bigr]}

\newcommand\lrsqpar[1]{\left[#1\right]}

\newcommand\Bigabs[1]{\Bigl|#1\Bigr|}

\newcommand\qw{^{-1}}
\newcommand\qww{^{-2}}
\newcommand\qq{^{1/2}}
\newcommand\qqw{^{-1/2}}

\newcommand\xxx{^{(x)}}
\newcommand\bS{\overline S}
\newcommand\bSx{\bS\xxx}
\newcommand\chL{\check L}
\newcommand\inton{\int_0^n}
\newcommand\ono{(0,n)}
\newcommand\on{[0,n]}
\newcommand\OFP{(\gO,\cF,P)}
\newcommand\wick[1]{{:}#1{:}}
\newcommand\HH{H^{\wick2}}
\newcommand\xB{T}
\newcommand\BB{Q}
\newcommand\Too{T_\infty}
\newcommand\gLx{\gL^*}

\newcommand\set[1]{\ensuremath{\{#1\}}}
\newcommand\bigset[1]{\ensuremath{\bigl\{#1\bigr\}}}

\newcommand\ceil[1]{\lceil#1\rceil}

\newcommand\punkt{.\spacefactor=1000}
\newcommand\iid{i.i.d\punkt}

\newcommand{\sumko}{\sum_{k=0}^\infty}

\newcommand{\sumk}{\sum_{k=1}^\infty}

\newcommand{\prodk}{\prod_{k=1}^\infty}

\newcommand\bbR{\mathbb R}

\newcommand\bbZ{\mathbb Z}

\newcommand\ga{\alpha}
\newcommand\gb{\beta}
\newcommand\gd{\delta}
\newcommand\gD{\Delta}

\newcommand\gG{\Gamma}

\newcommand\kk{\kappa}
\newcommand\gl{\lambda}
\newcommand\gL{\Lambda}
\newcommand\go{\omega}
\newcommand\gO{\Omega}
\newcommand\gs{\sigma}

\newcommand\gss{\sigma^2}
\newcommand\eps{\varepsilon}

\newcommand{\tend}{\longrightarrow}
\newcommand\dto{\overset{\mathrm{d}}{\tend}}

\newcommand\eqd{\overset{\mathrm{d}}{=}}

\newcommand{\refT}[1]{Theorem~\ref{#1}}
\newcommand{\refC}[1]{Corollary~\ref{#1}}
\newcommand{\refL}[1]{Lemma~\ref{#1}}
\newcommand{\refR}[1]{Remark~\ref{#1}}
\newcommand{\refS}[1]{\S\ref{#1}}

\newcommand{\refApp}[1]{Appendix~\ref{#1}}

\newcommand\oi{\ensuremath{[0,1]}}

\newcommand\intoi{\int_0^1}
\newcommand\Bigparfrac[2]{\Bigpar{\frac{#1}{#2}}}
\newcommand\lhs{left-hand side}
\newcommand\rhs{right-hand side}
\newcommand\ooo{[0,\infty)}
\newcommand\ntoo{\ensuremath{{n\to\infty}}}
\newcommand\intoo{\int_0^\infty}
\newcommand\xfrac[2]{#1/#2}
\newcommand\xij{_{ij}}

\newcommand\eg{e.g\punkt}

\newcommand\norm[1]{\|#1\|}
\newcommand\innprod[1]{\langle#1\rangle}

\newenvironment{romenumerate}[1][-10pt]{
\addtolength{\leftmargini}{#1}\begin{enumerate}
 }{\end{enumerate}}

 \newcommand\ttoo{\ensuremath{{t\to\infty}}}
 \newcommand\etta{\boldsymbol1} 
 
 \newtheorem{example}[theorem]{Example}
 
\newcommand\Cov{\operatorname{Cov}}
\newcommand\Bin{\operatorname{Bin}}
 \newcommand\nnn[1]{n_{#1}} 
 \newcommand\cA{\mathcal{A}} 
 \newcommand\cD{\mathcal{D}} 
 \newcommand\cE{\mathcal{E}} 
 \newcommand\cF{\mathcal F}
 \newcommand\WL{\hat L} 
 \newcommand\XV{V} 
 \newcommand\YW{W} 
 \newcommand\Ga{G_{\ga}} 
 \newcommand\Gao{H_{\ga}} 
 \newcommand\Br{Br} 
 \newcommand\gDx{\hat\gD_k} 
 \newcommand\gDp[1]{\gD(#1)} 
 \newcommand\gDnt{\gD_{\floor{2nt}}} 
 \newcommand\taue{\tau_{\eps}} 
 \newcommand\PrB{\widehat{\Pr}} 
 \newcommand\PrBx[1]{\PrB\bigsqpar{#1}} 
 \newcommand\PrBX[1]{\PrB\Bigsqpar{#1}}
 \newcommand\Prz[1]{\Pr[#1]}  
  
 \newcommand\PrX[1]{\Pr\Bigsqpar{#1}} 
 \newcommand\Doo{\cD[0,\infty)} 
 \newcommand\as{a.s.\spacefactor=1000} 
 \newcommand\pdf{probability density function} 
 \newcommand\ArXiv[1]{\url{ArXiv:#1}}
 \newcommand\cg{\theta} 

\begin{document}

\title{Preferential Attachment when Stable}
\author{Svante Janson}
\address{Department of Mathematics, Uppsala University, PO Box 480,
SE-751~06 Uppsala, Sweden}
\email{svante.janson@math.uu.se}

\author{Subhabrata Sen}
\address{Microsoft Research NE and Department of Mathematics, Massachusetts Institute of Technology, 
77 Massachusetts Avenue, Cambridge, MA- 02139, USA.}
\email{ssen90@mit.edu}
\author{Joel Spencer}
\address{ Department of Computer Science and Department of Mathematics, Courant Institute, New York University, 
Room 829, 251 Mercer St.
New York, NY 10012, U.S.A.}
\email{spencer@cims.nyu.edu}
\date{May 27, 2018}

\thanks{SJ was partly supported by the Knut and Alice Wallenberg Foundation}


\begin{abstract}
We study an urn process with two urns, initialized with a ball each. Balls are added sequentially, the urn being chosen independently with probability proportional to the $\alpha^{th}$ power $(\alpha >1)$ of the existing number of balls. We study the (rare) event that the urn compositions are balanced after the addition of $2n-2$ new balls.  We derive precise asymptotics of the probability of this event by embedding the process in continuous time. Quite surprisingly, a fine control on this probability may be leveraged to 
derive a lower tail Large Deviation Principle (LDP) for $L = \sum_{i=1}^{n} \frac{S_i^2}{i^2}$, where $\{S_n : n \geq 0\}$ is a simple symmetric random walk started at zero. We provide an alternate proof of the LDP via coupling to Brownian motion, and subsequent derivation of the LDP for a continuous time analogue of $L$. Finally, we turn our attention back to the urn process conditioned to be balanced, and provide a functional limit law describing the trajectory of the urn process.  
\end{abstract}

\maketitle


\section{Model and Summary of Results}\label{sec:model}

Consider two urns, each of which initially has one ball.  Balls come
sequentially.  When there are $i$ balls in the first urn and $j$ balls
in the second urn the next ball is placed in the first urn with probability
$i^{\ah}/(i^{\ah}+j^{\ah})$ and into the second urn with probability
$j^{\ah}/(i^{\ah}+j^{\ah})$.  Here $\ah$ is a constant.  In this work
we shall consider only $\ah > 1$, though we note that the case $\ah=1$
is the classic P\'olya Urn Model; see Remark \ref{Ralpha}.  
Equivalently, we define a Markov Chain
$\{(X_n,Y_n): n \geq 2\}$
with states $\mathbb{N}\times \mathbb{N}$ and intial state
$(X_2,Y_2)=(1,1)$, and 
transition probabilities 
\begin{align}
\P[(X_{n+1} , Y_{n+1} ) = (i+1, j) \mid (X_n, Y_n) = (i,j) ] 
&= \frac{i^{\alpha}}{i^{\alpha}+ j^{\alpha}},
\label{mc1}
\\
\P[ (X_{n+1}, Y_{n+1}) = (i, j+1) \mid (X_n, Y_n) = (i,j) ] 
&= \frac{j^{\alpha}}{i^{\alpha}+ j^{\alpha}}.
\label{mc2}
\end{align}
(We have chosen the notation such that $X_n+Y_n=n$; we thus start with
$n=2$.)

In this model the rich get richer.  It is highly unstable, as demonstrated
by Theorems \ref{t2} and \ref{t3} below.  We examine the (rare) event that
the urn populations remain stable.  For definiteness we concentrate on the
event that state $(n,n)$ is reached, i.e., $(X_{2n},Y_{2n})=(n,n)$,
which we denote $BINGO(n,n)$.  Continuous time, as described in 
\S\ref{sec:cont_time}, is a powerful method which yields ``Book Proofs" of
asymptotic stability probabilities.  

In an apparently unrelated direction, 
let $\xi_1,\ldots,\xi_n=\pm 1$, uniformly and independently, and set
$S_t = \sum_{i=1}^t \xi_i$.  That is, $S_t$ is the position of the
standard simple random walk at time $t$.  We set
\beq\label{1a}
L = L_n = \sum_{i=1}^n \frac{S_i^2}{i^2}. \eeq
A sequence  $\xi_1,\ldots,\xi_{2n-2}=\pm 1$ corresponds to a path
from $(1,1)$, with $\pm 1$ representing horizontal and vertical
moves respectively.  
(This is the case $\ah=0$ of the process defined above, when all steps are
independent.) 
$L$ is then a key statistic for measuring
how far the path strays from the main diagonal.  
As $\E[S_i^2]=i$, 
\begin{equation}
  \label{eq:EL}
  \E[L_n]=\sum_{i=1}^n i^{-1} = \ln n + O(1),
\end{equation}
and, as will be seen later, $L_n$ is typically about $\ln n$.
Our main 
concern is with the lower tail of the distribution of $L_n$, and we prove
the following.
(See \refT{TU} for a corresponding result for the upper tail.)
\begin{theorem}\label{t1}
For any fixed $c \in (0,1)$,
\beq\label{2a} \Pr[L_n \leq c\ln n] = e^{-(K(c)+o(1))\ln n}  \eeq
with
\beq\label{3a} K(c) =\frac{(1-c)^2}{8c}  .
\eeq
\end{theorem}

The variable $L_n$ has an intriguing behavior.  
Parametrizing $S_i=\sqrt{i}N_i$,
the $N_i$ are asymptotically (in $i$) standard Gaussian and $L=\sum N_i^2/i$.
The harmonic series suggests a logarithmic scaling, $t=\ln i$.
Note that under this scaling we have strong correlation when $t,t'$ are
close which fades as the distance increases.  
That is, $S_i,S_{i\lam}$ are closely
correlated when $\lam$ is close to one and have positive asymptotic
correlation for any fixed $\lam$, but that correlation approaches zero
as $\lam$ approaches infinity.  

We give two very different arguments for Theorem \ref{t1}.  In
\S \ref{sec:basic},\ref{sec:lower},\ref{sec:upper} we employ the 
Markov process $(X_n,Y_n)$ and the
continuous time argument for it in \refS{sec:cont_time} to derive the
Laplace transform 
of $L$ and from that deduce the large deviation Theorem \ref{t1}.
In \S \ref{sec:brownianapprox}--\ref{sec:brownian_analysis}, we provide a more traditional proof, which turns out to be quite challenging. We couple the random walk to a standard Brownian motion via the celebrated KMT coupling, and establish that it suffices to derive the corresponding LDP for a Brownian analogue of $L$. Subsequently, we derive the LDP by the general theory for quadratic functionals of Brownian motion. 

Theorem \ref{t1} establishes a rigorous lower tail Large Deviation Principle
(LDP)  for a quadratic functional of the Simple Random Walk. Large
deviations for non-linear functions of $\{\pm 1\}$ variables has been an
active research area in recent years. In a breakthrough paper,  Chatterjee
and Dembo \cite{chatterjee2016ldp} initiated a systematic study of LDPs of
non-linear functionals of $\{\pm 1\}$ variables. The theory was subsequently
extended by Eldan \cite{eldan}, and has been applied to numerous problems in
probability and combinatorics (see
e.g. \cite{bglz2017,bgsz2018,eldangross2017a, eldangross2017b, lz2017}). We
emphasize that Theorem \ref{t1} does not follow using the general theory
established in these prior works, and that our approaches are entirely
different.

In \refS{section:conditionallaw} we return to the
Markov process $(X_k,Y_k)$ defined above, conditioned on the rare event
$BINGO(n,n)$,
and 
examine the typical path from $(1,1)$
to $(n,n)$.  
We define $\gD_k:=X_k-Y_k$, so that
\begin{equation}\label{XYgD}
  (X_k,Y_k)=\Bigpar{\frac{k+\gD_k}2,\frac{k-\gD_k}2},
\qquad k\ge2,
\end{equation}
and define, for completeness, $\gD_0=\gD_1:=0$.
(For typographical reasons,
we sometimes write $\gD(k)$.) 
Note that the event $BINGO(n,n)$ can be written $\gD_{2n}=0$.
We provide a functional limit law which
shows that  conditioned on $BINGO(n,n)$,
$\gD_k$ is typically of order $\sqrt n$ for $2<k<2n$, and that suitably
rescaled, $(\gD_k)_2^{2n}$ converges to a distorted Brownian bridge.
\begin{theorem}\label{Toi}
  Let $\Ga(t)$, $t\in\oi$, be the continuous Gaussian process with mean $0$
  and covariance function
  \begin{equation}\label{toi1}
    \Cov\bigpar{\Ga(s),\Ga(t)}
=\frac{2}{2\ga-1} s^\ga \bigpar{t^{1-\ga}-t^\ga}
,\qquad 0\le s\le t\le 1.
  \end{equation}
Then, 
as \ntoo, 
conditioned on $BINGO(n,n)$ (i.e.,  $\gD_{2n}=0$), 
  \begin{equation}\label{toi2}
    n\qqw \gD_{\floor{2nt}} 
\dto 
\Ga(t),
\qquad t\in\oi,
  \end{equation}
in $\cD\oi$ with the Skorohod topology. 
\end{theorem}

$\Ga(t)$ can be constructed from a standard Brownian bridge $\Br(t)$ as
\begin{equation}\label{G-Br}
  \Ga(t):=(\ga-1/2)\qqw t^{1-\ga}\Br\bigpar{t^{2\ga-1}},
\end{equation}
with $\Ga(0):=0$.
Related constructions from a  Brownian motion  are given in
\eqref{G-B1}--\eqref{G-B2}.

We give also a version of this theorem for $k=o(n)$.
Now a Brownian motion $B(t)$ appears instead of a Brownian bridge.

\begin{theorem}  \label{Too}
  Let $m_n\to\infty$ be real numbers with $m_n=o(n)$.
Then, 
as \ntoo, 
conditioned on $BINGO(n,n)$ (i.e.,  $\gD_{2n}=0$), 
  \begin{equation}\label{too}
    m_n\qqw \gD_{\floor{m_nt}} 
\dto 
\Gao(t):=(2\ga-1)\qqw t^{1-\ga} B\bigpar{t^{2\ga-1}},
\qquad t\in\ooo,
  \end{equation}
in $\Doo$ with the Skorohod topology. 
\end{theorem}

In particular, taking $m_n$ integers and $t=1$, it follows that for any 
integers $m=m_n\to\infty$ with $m=o(n)$, conditioned on $BINGO(n,n)$,
  \begin{equation}\label{toox}
    m\qqw \gD_{m}
\dto N\Bigpar{0,\frac{1}{2\ga-1}}.
  \end{equation}
For $m<n$ with $m=\Theta(n)$, we obtain from \refT{Toi} a similar result
with a correction factor 
for the variance. Thus, $\gD_m$ is typically of order $\sqrt m$ for
$m<n$.

Under a suitable logarithmic scaling, we have in the limit a stationary
Ornstein--Uhlenbeck process, defined by \eqref{tou2} below.

\begin{theorem}\label{thm:oulimit}
Fix any sequence $t_n$ such that $t_n \to \infty$ and $\log n - \log t_n \to
\infty$. 
Then we have, as $n \to \infty$, conditional on $BINGO(n,n)$, 
\begin{equation}\label{tou1}
 e^{-(s+t_n)/2}\Delta\bigpar{\floor{e^{s+ t_n}} }
\dto Z(s),
\qquad -\infty < s <\infty,
\end{equation}
in $\mathcal{D}(-\infty, \infty)$ with the Skorokhod topology, 
where $Z(s)$ is a centered Gaussian process  with covariance function
\begin{equation}\label{tou2}
\E[Z(s) Z(t)] = \frac{1}{2\alpha-1} e^{- (\alpha - \frac{1}{2}) |s -t | },
\qquad s,t\in\bbR.  
\end{equation}
\end{theorem}

\begin{remark}\label{Ralpha}
As said above, we consider in this paper only $\ga>1$, which is necessary
e.g.\ for \refT{t2}. However, it would be interesting to study also
$\ga\in\oi$,
when $(X_k,Y_k)$ behaves quite differently.
Note that $\ga=0$ yields $\gD_k$ as a simple random walk, and then it is
well-known that \eqref{toi2} holds with $G_0(t):=\sqrt2 \Br(t)$,
see e.g.\ \cite[Theorem 24.1]{Billingsley}. 
(The factor $\sqrt2$ is because of our choice of normalization.)
Furthermore, for $\ga=1$, when as said above $(X_k,Y_k)$ is the classical
P\'olya urn, it is well-known that the increments are exchangeable, and
thus, conditioned on $BINGO(n,n)$, all paths to $(n,n)$ have the same
probability. 
(This can be seen from \eqref{11a}--\eqref{13b} below, noting that for
$\ga=0$ or $\ga=1$, $FIT_i$ in \eqref{9a} is constant 1.)
I.e., conditioned on $BINGO(n,n)$, $\ga=1$ and $\ga=0$
coincide, and thus \eqref{toi2} holds for $\ga=1$ too,  with 
$G_1(t) =G_0(t)=\sqrt2 \Br(t)$. 
Note further that this agrees with \eqref{toi1} and \eqref{G-Br}
for $\ga=1$ (but not for
$\ga=0$). 
Similarly, \eqref{too} and \eqref{tou1}--\eqref{tou2} hold for $\ga=1$, with
$H_1(t)=B(t)$. 
It would be interesting to find an analogue of \refT{Toi} for
$0<\ga<1$. 
\end{remark}

\section{Continuous Time}
\label{sec:cont_time}
\noindent
We examine the Markov Chain defined in \S \ref{sec:model} \eqref{mc1}--\eqref{mc2} with
initial state $(1,1)$.

\begin{define}\label{defbingo}
$BINGO(i,j)$ denotes the event that state $(i,j)$ is reached.
\end{define}

This preferential attachment model is best attacked (see Remark \ref{rem1}) 
via  continuous
time.  Let $\XV_i,\YW_i$, $i\geq 1$, denote exponential distributions
with rate parameter $i^{\ah}$. That is, $\XV_i,\YW_i$ have \pdf{} 
$\lam e^{-\lam x}$
with $\lam=i^{\ah}$.  The $\XV_i,\YW_i$ are all chosen mutually independently. 
Begin the urn model, as before, with each urn having one ball.  Begin time
at zero.  When an urn has $i$ balls it waits time $\XV_i$ until it receives
its next ball. The forgetfulness property of the exponential distribution
(plus a little calculus) gives that when the bins have $i,j$ balls respectively
the probability that $\XV_i < \YW_j$ is 
$i^{\ah}/(i^{\ah}+j^{\ah})$ as desired.  
This leads to a remarkable theorem
(\cite{davis}, but see Remark \ref{rem1})
with what is surely a Proof from The Book.

\begin{theorem}\label{t2} With probability $1$ one of the bins gets all but
a finite number of the balls.
\end{theorem}

\begin{proof}
Let $\XV=\sum_{i=1}^{\infty}\XV_i$,
$\YW=\sum_{i=1}^{\infty}\YW_i$.  As $\sum i^{-\ah}$ is finite
(here using that $\ah > 1$) both $X$ and $Y$ are finite a.s.  As
the distribution is nonatomic, $X\neq Y$ a.s.  Say $X < Y$.  Then
bin one receives all its ball before bin two does.  When bin
one has all its balls the process stops (a countable number of
balls have been placed) and bin two only has a finite number of
balls.  
\end{proof}

\begin{corollary}\label{cor1} 
\beq\label{4a} \lim_{M\ra\infty}\lim_{k\ra\infty}\sum_{i=M}^{k-M}\Pr[BINGO(i,k-i)]=0 \eeq
\end{corollary}

\begin{proof}
For $M$ fixed let $FENCE(k)$ denote the disjunction of the $BINGO(i,k-i)$
over $M\leq i\leq k-M$;
this is the event that at the time that there are $k$ balls in the urns,
there is at least $M$ balls in each urn.
As these $BINGO(i,k-i)$ are tautologically disjoint (a path can only
hit one state with a given sum of coefficients), $\Pr[FENCE(k)]$ is given by the sum in (\ref{4a}).
Tautologically $FENCE(k)$ implies $FENCE(k')$ for all $k'\geq k$ as once a path hits the fence at
$k$ it cannot escape the fence at $k'$.  Thus the disjunction of all $FENCE(k)$ has probability
$\lim_{k\ra\infty}$ of the sum.  But the disjunction is (again tautologically!) the event that
both bins eventually get at least $M$ balls.  From Theorem \ref{t2}, this has limiting value
(in $M$) of zero.  
\end{proof}

\begin{remark}\label{rem1} The use of continuous time appears to be due to Herman Rubin, as
attributed by Burgess Davis in \cite{davis}.  A thorough study of preferential attachment (in a
far more general setting) via continuous time was given in the Ph.D. thesis of Roberto Oliveira,
under the supervision of the senior author (JS).  Many of the results of Oliveira's thesis
are given in \cite{spenceroliveira}. 
Theorem \ref{t2} and Corollary \ref{cor1} provided the
orginal motivation for our current research. 
The senior author  searched for a combinatorial proof, appropriately counting paths with their respective
probabilities, for Corollary \ref{cor1}.  This in turn led to attempts to estimate
$BINGO(i,j)$ without using continuous time.  Somewhat surprisingly, one result is in
the opposite direction.  The estimates on $BINGO(i,j)$ given by continuous time have
given a quite roundabout argument for
the large deviation results for the random variable $L$ given by Theorem \ref{t1}.
\end{remark}

Continuous time gives us excellent asymptotics on $BINGO$.  We first provide the (tautological) bridge between
continuous time and $BINGO$.

\begin{theorem}\label{newj1}  Set
\beq\label{defdeltaeasy} 
\Delta := \sum_{s=1}^{i-1}\XV_i - \sum_{t=1}^{j-1}\YW_j  \eeq
Then $BINGO(i,j)$ occurs iff either $0 \le \Del < \YW_j$
or $0 \le -\Del < \XV_i$. 
\end{theorem}

\begin{proof}
$\Delta$ is the time difference between when urn one receives its $i$-th ball and
urn two receives its $j$-th ball.  Suppose $\Del \le 0$.  At time $T=\sum_{s=1}^{i-1}$ $\XV_i$ urn
one receives its $i$-th ball.  Urn two will receive its $j$-the ball at time $T-\Del$.  $BINGO(i,j)$
occurs when urn one has not yet received its $i+1$-st ball, which it does at time $T+\XV_i$.  This
occurs iff $\XV_i > -\Del$.  The case $\Del \ge 0$ is similar.  
\end{proof}

\begin{theorem}\label{t3}  There is a positive constant $\beta$, dependent only on $\ah$, so that
when $i,j\ra\infty$
\beq\label{5a}  \Pr[BINGO(i,j)] \sim \beta[i^{-\ah} + j^{-\ah}].  \eeq
In particular,
\beq\label{bingo}  \Pr[BINGO(n,n)] \sim 2\beta n^{-\ah} .  \eeq
\end{theorem} 

\begin{remark}\label{remj1}   
Set $\Del^{\dagger} := \sum_{i=1}^{\infty} (\XV_i-\YW_i)$.  Basically $\Del$ is estimated by $\Del^{\dagger}$ and $\beta$ is
the \pdf{} of $\Del^{\dagger}$ at $0$.  $\YW_j$ is almost always $o(1)$ (as
$j\ra\infty$) so that $0 \le \Del < \YW_j$ should occur with 
asymptotic probability $\beta \E[\YW_j]  = \beta j^{-\ah}$.
However  the validity of the approximation is 
nontrivial and has
forced our somewhat technical calculations. 
\end{remark}

\begin{proof}
We analyze $\P[0\le \Delta < \YW_{j}]$ as $ i,j \to \infty$. The analysis of the
other term is similar, and is thus omitted. 
 Note that $\Delta=\Delta\xij$ is the sum of independent random variables,
 each with a density with respect to Lebesgue measure. 
Thus $\Delta\xij$ has a \pdf, which we denote as $f\xij$. 

We will use characteristic functions to study the density $f\xij$. We set
$\phi\xij(t) = \E[ \exp(\ii t \Delta\xij)]$. Upon direct computation, we have, 
\begin{align}
\phi\xij(t) = \prod_{k=1}^{i-1} \frac{k^{\alpha}}{k^{\alpha} - \ii t}
\prod_{k=1}^{j-1} \frac{k^{\alpha}}{k^{\alpha} + \ii t }.  \nonumber
\end{align}

Using the Fourier inversion theorem, the density may be related to the
characteristic function. Thus,
provided $i+j\ge4$, $\phi\xij(t)$ is integrable
and then 
\begin{align}
f\xij(x) 
= \frac{1}{2 \pi} \int_{-\infty}^{\infty} e^{- \ii t x} \phi\xij(t) \dd t 
\label{eq:density}
= \frac{1}{\pi} \int_{0}^{\infty} \cos(tx) \phi\xij(t) \dd t. 
\end{align}

In particular, 
\begin{align}
f\xij(0) = \frac{1}{2 \pi} \int_{-\infty}^{\infty} \phi\xij(t) \dd t. \nonumber
\end{align}

Further, note that for each $t \in \reals$ as $i,j \to \infty$, $\phi\xij(t)  \to \prod_{k=1}^{\infty} \frac{k^{2\alpha}}{k^{2\alpha} + t^2}$ and thus by Dominated Convergence, 
\begin{align}
f\xij(x) \to \frac{1}{2\pi} \int_{-\infty}^{\infty}  \cos(tx) \prod_{k=1}^{\infty} \frac{k^{2\alpha}}{k^{2\alpha} + t^2}\dd t =: f_{\infty}(x). \label{eq:density_infty}
\end{align}

This establishes the pointwise convergence of the density. 
Moreover, the same argument shows that for any convergent sequence
$x_{ij}\to x$, $f\xij(x_{ij})\to f_\infty(x)$. 
In particular, we define $\beta := f_{\infty}(0)$. 
 We have, since $\YW_j$ is independent of $\Del\xij$, 
and has the same distribution as $j^{-\alpha}\YW_1$,
\begin{align}
\P[ 0 \leq \Delta\xij < \YW_{j} ] &= \E\Bigl[ \int_{0}^{\YW_{j}} f\xij(z) \dd z \Bigr] 
= \E\Bigl[j^{-\alpha} \int_{0}^{\YW_{1}} f\xij(j^{-\alpha}z)\, \dd z \Bigr] 
\end{align}
and hence, by dominated convergence,
\begin{align}
j^{\alpha} \P[ 0 \leq \Delta\xij < \YW_{j} ]   
\to \E\Bigl[ \int_{0}^{\YW_{1}} f_{\infty}(0) \,\dd z \Bigr] 
= f_{\infty}(0) 
\end{align}
as $i,j \to \infty$. 
This establishes the required asymptotics of $\P[ 0 \leq \Delta\xij < \YW_{j} ]   $.  
\end{proof}

\section{More Continuous Time}
\label{sec:cont_time2}
We generalize $BINGO$ to allow for arbitrary initial states.

\begin{define}\label{defbingoext}
$BINGO(a,b;c,d)$ denotes the event
that the Markov Chain $(X_k,Y_k)$ given by \eqref{mc1}--\eqref{mc2}
with initial state $(a,b)$ reaches state
$(c,d)$. 
\end{define}

As before, let $\XV_i,\YW_j$ be exponentials at rate $i^{\ah},j^{\ah}$ but
now restrict to $i\geq a, j\geq b$.  Again we have a bridge.

\begin{theorem}\label{newj2} 
Let $a\le c$ and $b\le d$, and
set
\beq\label{defdelta} 
\Del := \sum_{i=a}^{c-1}\XV_i - \sum_{j=b}^{d-1} \YW_j . \eeq
$BINGO(a,b;c,d)$
occurs if and only if either $0 \le -\Del < \XV_c$ or 
$0 \le \Del < \YW_d$.
\end{theorem}
\begin{proof}
The same as  for Theorem \ref{newj1}.  
\end{proof}

Continuous time gives the asymptotics of $BINGO$ for a wide variety of
the parameters.  
We derive accurate estimates for various $BINGO$ events, which will be used
in our subsequent discussions.

We begin with the simplest case, starting and ending on the diagonal.
(Cf.\ \eqref{bingo}, when starting at $(1,1)$.)
\begin{theorem}\label{bingoaann} 
For any sequence $A = A(n) \to \infty$ with  $A(n) = o(n)$, 
as $n \to \infty$, 
\beq\label{7.6a}  
\Pr[BINGO(A,A;n,n)] \sim 
\Bigl(\frac{2\ah-1}{\pi}\Bigr)^{1/2} A^{(2\ah-1)/2}n^{-\ah}  .\eeq
\end{theorem}


In the proof of \refT{t1}, we use only the weaker 
\beq\label{7.2a} \Pr[BINGO(A,A;n,n)] = n^{-\ah+o(1)},
\qquad A=n^{o(1)} . \eeq

Before proving \refT{bingoaann}, we state a generalization, where we allow
initial and final points that are off the diagonal (but not too far away; we
consider only what will turn out to be the typical cases, see Theorems
\ref{Toi} and \ref{Too}). We also, for later use in
\refS{section:conditionallaw}, allow $A=\Theta(n)$ as long as
$n-A=\Theta(n)$
(and in this connection
we change the notation from $n$ to $B$).

\begin{theorem}\label{Tbingo}
 Fix $M >0$, and $\cg>1$. 
Then,
uniformly for all  $A,B,\gL,\gG\in\frac12\bbZ$ 
with $A\pm\gL\in\bbZ$, $B\pm\gG\in\bbZ$ 
such that $A>0$, $B\ge \cg A$, 
$|\gL| \leq M \sqrt{A}$, 
and $|\gG| \leq M \sqrt{B}$, 
\begin{align}\label{near_multiple}
& \Pr[ BINGO(A+ \gL,A - \gL; B + \gG, B - \gG)] 
\nonumber\\&
\quad=
\bigpar{1+o_{A}(1)}\sqrt{\frac{2\ga-1}{\pi}}
\frac{A^{\ga-1/2}}{B^\ga\sqrt{1-(A/B)^{2\ga-1}}}
\nonumber\\&
\qquad
\qquad
\times
\exp\biggpar{-\frac{2\ga-1}{1-(A/B)^{2\ga-1}}
  \Bigpar{\frac{\gL}{A^{1/2}}-\frac{\gG}{B^{\ga\vphantom{/2}} A^{1/2-\ga}}}^2}
,\end{align}
where $o_A(1)$ is a quantity that tends to $0$ as $A\to\infty$, uniformly in
the other variables; i.e., $|o_A(1)|\le \eps(A)$ for some function
$\eps(A)\to0$ as $A\to\infty$.
\end{theorem}

\begin{remark}\label{Rgauss}
 The right-hand side of \eqref{near_multiple}, omitting the
$o_{A}$ term, is the density function at $\gG$ for a normal distribution
$N(\mu,\gss)$ with parameters
\begin{align}
\mu&=(B/A)^\ga \gL,
\\
  \gss&=
\frac{ A^{1-2\ga}B^{2\ga}\bigpar{1-(A/B)^{2\ga-1}}}{2(2\ga-1)}.
\end{align}
\end{remark}

The two main cases of interest to us are $A\ll B=n$, as in \refT{bingoaann},
and $B/A$ constant (at least up to rounding errors). For convenience, we
state immediate corollaries covering these cases.

\begin{corollary}\label{Cbingo0}
Suppose $A=A(n)\to\infty$ with $A=o(n)$.
Then,
for all  $\gL=\gL(n)$ and $\gG=\gG(n)$
with $\gL=O\bigpar{\sqrt A}$ 
and $\gG=O\bigpar{\sqrt n}$, 
\begin{multline}\label{near_multiple0}
 \Pr\bigsqpar{ BINGO(A+ \gL,A - \gL; n + \gG, n - \gG)}
\\
\sim
\sqrt{\frac{2\ga-1}{\pi}}
\frac{A^{\ga-1/2}}{n^\ga}
\exp\biggpar{-\bigpar{2\ga-1}
  \frac{\gL^2}{A}}
.\end{multline}
\end{corollary}

In \refS{sec:lower}, we use  only the rougher asymptotics, 
extending \eqref{7.2a}: 
\beq\label{nearrough}  
\Pr[BINGO(A+\Lam,A-\Lam;n+\Gam,n-\Gam)] = n^{-\ah+o(1)},
\qquad A=n^{o(1)} . 
\eeq

\begin{corollary}\label{Cbingo2}
Suppose $A=A(n)$ and $B=B(n)$ with $A\to\infty$ and $B/A\to\cg>1$ as \ntoo. 
Then,
for all  $\gL=\gL(n)$ and $\gG=\gG(n)$
with $\gL=O\bigpar{\sqrt A}$ 
and $\gG=O\bigpar{\sqrt B}$, 
\begin{align}\label{near_multiple2}
& \Pr[ BINGO(A+ \gL,A - \gL; B + \gG, B - \gG)] 
\nonumber\\&
\quad
\sim
\sqrt{\frac{2\ga-1}{\pi}}
\frac{A^{\ga-1/2}}{B^\ga\sqrt{1-\cg ^{1-2\ga}}}
\exp\biggpar{-\frac{(2\ga-1)}
{\bigpar{1-\cg^{1-2\ga}}A}  \bigpar{\gL-\cg^{-\ga}\gG}^2}
.\end{align}
\end{corollary}

The proofs of Theorems \ref{bingoaann} and \ref{Tbingo}
are similar to the  proof of \refT{t3}.
We begin with a proof of the simpler \refT{bingoaann}, to show the main
features of the proof. We then show the modifications needed for the more
general \refT{Tbingo}.

\begin{remark}\label{remj2}  
Theorem \ref{bingoaann} is
basically a  local CLT for $\Del$, cf.\ Remark \ref{remj1}.
Set 
$\Del_n:= \sum_{k=A}^{n-1}(\XV_k-\YW_k)$.  $\Del_n$
is asymptotically Gaussian with mean $\mu=0$ and variance
\beq\label{eq:var} 
\sig_n^2 := \Var[\gD_n]
=2\sum_{k=A}^{n-1}k^{-2\ah} \sim \frac{2}{2\ah-1}A^{1-2\ah} . \eeq
Approximating $\Del_n$ by this Gaussian, it has \pdf{} at zero asymptotically
$(2\pi)^{-1/2}\sig^{-1}=((2\ah-1)/4\pi)^{1/2}A^{(2\ah-1)/2}$.
The probability that $0 \le -\Del_n < \XV_n$
is then $\sim \E[\XV_{n}]= n^{-\ah}$ times this, and $\Pr[BINGO]$ is twice
that.

Note also that in \refC{Cbingo0}, the probability has an extra factor of 
$\exp[-(2\ah-1)(\gL/\sqrt A)^2]$ over the basic $\lam=0$ case of Theorem
\ref{bingoaann}. 
Rougly, while $\Del$ is still asymptotically Gaussian, the mean has moved
$\sim 2\Lam A^{-\ah}$ from zero.  

As before,  the validity of the approximations are 
 nontrivial and has
forced our somewhat technical calculations. 
\end{remark}

\begin{proof}[Proof of \refT{bingoaann}]
 We assume $A\le n-1$ and define, as in \refR{remj2}, 
 $\Delta_n := \sum_{k=A}^{n-1} (\XV_k - \YW_k)$. 
We note that $\gD_n$ has is centered with
variance $\gss_n$ given by \eqref{eq:var}, and, by the same argument as in the
proof of \refT{t3}, it has a \pdf, which we denote as $f_n$. 

We set the characteristic function 
$\phi_n(t) := \E[\exp(\ii t \Delta_n) ] $ and note that, by direct computation, 
\begin{align}
\phi_n(t) = \prod_{k=A}^{n-1} \frac{k^{2\alpha}}{k^{2 \alpha} + t^2}. 
\label{eq:cf_exp}
\end{align}
As in our earlier analysis, we 
use the Fourier inversion formula to conclude
that 
\begin{align}
f_n(x) = \frac{1}{2\pi} \int_{-\infty}^{\infty} e^{-\ii tx} \phi_n(t) \dd t. 
\label{eq:cf}
\end{align}
This again implies that $f_n(x) \leq f_n(0)$.
Note that \eqref{eq:cf} implies 
\begin{align}
\sigma_n f_n(0) = \frac{\sigma_n}{\pi} \int_0^{\infty} \phi_n(t) \dd t. \nonumber
\end{align}
Using the change of variables $v = \sigma_n t$, we have, 
by \eqref{eq:cf_exp},
\begin{align}\label{byx}
\sigma_n f_n(0) 
= \frac{1}{\pi} \int_0^{\infty} \phi_n\Big(\frac{v}{\sigma_n}  \Big) \dd v 
= \frac{1}{\pi} \int_0^{\infty} \prod_{k=A}^{n-1} \frac{k^{2 \alpha}}{k^{2 \alpha} + \frac{v^2}{\sigma_n^2}} \dd v. 
\end{align}

We use the dominated convergence theorem, and begin by noting that for $k\ge
A$, we have, using \eqref{eq:var}
and letting $C$ and $c$ denote unspecified positive constants,
\begin{align}\label{bya}
k^{-2\ah}  \frac{v^2}{\sigma_n^2} 
\le  A^{-2\ah}  \frac{v^2}{\sigma_n^2} 
=
O(v^2/A).
\end{align}
In particular, for any fixed real $v$, recalling again \eqref{eq:cf_exp}
and \eqref{eq:var},
\begin{align}\label{byb}
\ln  \phi_n \Big(\frac{v}{\sigma_n} \Bigr) 
= 
-\sum_{k=A}^{n-1}\ln\Bigl(1+  \frac{v^2/\sigma_n^2}{k^{2 \alpha}}\Bigr)
\sim
-\sum_{k=A}^{n-1}  \frac{v^2/\sigma_n^2}{k^{2 \alpha}}
=-\frac{v^2}{2}
\end{align}
and thus
\begin{align}\label{byc}
  \phi_n \Big(\frac{v}{\sigma_n} \Bigr) 
\to
e^{-{v^2}/2}.
\end{align}
Furthermore, if $|v|\le\sqrt A$, then \eqref{bya} shows
$k^{-2\ah}v^2/\sigma_n^2=O(1)$ and thus, for some $c>0$,
$\ln\bigl(1+k^{-2\ah}v^2/\sigma_n^2\bigr)\ge ck^{-2\ah}v^2/\sigma_n^2$
and, similarly to \eqref{byb},
$\ln  \phi_n ({v}/{\sigma_n}) 
\le -c\frac{v^2}{2}
$ 
and thus
\begin{align}\label{byd}
  \phi_n \Big(\frac{v}{\sigma_n} \Bigr) 
\le
e^{-c{v^2}/2},
\qquad |v|\le\sqrt A.
\end{align}
If $|v|>\sqrt A$, we instead have, when $k\le 2A$,
using again \eqref{eq:var},
\begin{align}\label{bye}
  \frac{k^{2\ah}}{k^{2\ah}+v^2/\sigma_n^2}
\le   \frac{(2A)^{2\ah}}{(2A)^{2\ah}+A/\sigma_n^2} 
\le   \frac{2^{2\ah}}{2^{2\ah}+c_1} 
=c_2<1.
\end{align}
 Thus, crudely, by \eqref{eq:cf_exp} and \eqref{eq:var},
for large enough $n$ and $|v|>\sqrt A$,
\begin{align}
  \phi_n \Big(\frac{v}{\sigma_n} \Bigr) 
&
\le\prod_{A}^{2A}  \frac{k^{2\ah}}{k^{2\ah}+v^2/\sigma_n^2}
\le   \frac{A^{2\ah}}{A^{2\ah}+v^2/\sigma_n^2}
\prod_{A+1}^{2A} c_2
\le  \frac{A^{2\ah}\sigma_n^2}{v^2}c_2^A
.
\label{byf}
\end{align}
For convenience, we combine \eqref{byd} and \eqref{byf} into the (far from
sharp) estimate, valid for large $n$ and all $v$,
\begin{align}
  \phi_n \Big(\frac{v}{\sigma_n} \Bigr) 
=O\Bigpar{\frac{1}{1+v^2}}
.
\label{byff}
\end{align}

Consequently, dominated convergence yields, using 
\eqref{byx}, \eqref{byc} and \eqref{byff},
\begin{align}\label{byg}
\sigma_n f_n(0) 
=
\frac{1}{\pi} \int_{0}^\infty\phi_n\Bigpar{\frac{v}{\sigma_n} } \dd v 
\to \frac{1}{\pi} \int_0^{\infty} e^{-v^2/2} \dd v 
=\frac{1}{\sqrt{2\pi}}.
\end{align}
Moreover, for any sequence $x_n=o(\sigma_n)$, we obtain in the same way from
\eqref{eq:cf}
\begin{align}\label{byh}
\sigma_n f_n(x_n) 
= \frac{1}{2\pi} \int_{-\infty}^{\infty} 
e^{{-\ii vx_n}/{\sigma_n}}
\phi_n \Big(\frac{v}{\sigma_n}  \Big) \dd v 
\to \frac{1}{2\pi} \int_{-\infty}^{\infty} e^{-v^2/2} \dd v 
=\frac{1}{\sqrt{2\pi}}.
\end{align}

We use again that $\YW_{j}$ has the same distribution as $j^{-\ah}\YW_1$ and
obtain
\begin{align}\label{byi}
\P[ 0\le  \Delta_n < \YW_{n}]  
= \E \int_0^{\YW_{n}} f_n(x) \dd x
=n^{-\ah} \E \int_0^{\YW_{1}} f_n\Bigl(\frac{y}{n^{\ah}}\Bigr) \dd y.
\end{align}
Recall also that $f_n(x)\le f_n(0)$ for every $x$, and thus \eqref{byg}
implies that 
$\sigma_n f_n(x)$ is uniformly bounded for all $n$ and $x$.

Since, for every fixed $y$, 
$y/n^{\ah}=o(\sigma_n)$ by \eqref{eq:var}, we can use \eqref{byh}
and dominated convergence (twice) in \eqref{byi} and obtain
\begin{align}
n^{\ah}\sigma_n\P[ 0\le  \Delta_n < \YW_{n}]  
&=\E \int_0^{\YW_{1}} \sigma_nf_n\Bigl(\frac{y}{n^{\ah}}\Bigr) \dd y
 \to \E\Bigsqpar{\frac{\YW_{1}}{\sqrt{2\pi}}}
= \frac{1}{\sqrt{2\pi}}.
\label{byj}
\end{align}
Thus,
\begin{align}
\P[ 0\le  \Delta_n < \YW_{n}]  
\sim
 \frac{1}{\sqrt{2\pi}}\sigma_n^{-1}n^{-\ah}
= \Bigl(\frac{2\ah-1}{4\pi}\Bigr)^{1/2} A^{(2\ah-1)/2}n^{-\ah}.
\label{byk}
\end{align}
The probability $\P[ 0\le  -\Delta_n < \XV_{n}]  $ is the same, and \eqref{7.6a}
follows. 
\end{proof}

\begin{proof}[Proof of \refT{Tbingo}]
The proof is similar to that of Theorem \ref{bingoaann} and thus we detail
only the novelties and omit some parts which are similar to the earlier
proof. 

Let $p(A,B,\gL,\gG)$
denote the left-hand side of \eqref{near_multiple},
and let $q(A,B,\gL,\gG)$ the right-hand side without the factor $1+o_A(1)$.
First, 
note that if the asserted uniform estimate does not hold, then there
exist $\eps>0$ and
 $A=A(n)\to\infty$, $B=B(n)$, $\gL=\gL(n)$ and $\gG=\gG(n)$ that
satisfy the conditions such that $|p(A,B,\gL,\gG)/q(A,B,\gL,\gG)-1|>\eps$
for every $n$. By selecting a subsequence, we may furthermore
assume that 
\begin{align}\label{extra}
A/B\to\zeta,\qquad 
\gL/\sqrt A\to \gl,\qquad 
\gG/\sqrt B\to \gam, 
\end{align}
for some $\zeta\in[0,1)$ and $\gl,\gam\in\bbR$.
Hence, to obtain the desired contradiction, it suffices to prove 
that $p(A,B,\gL,\gG)\sim q(A,B,\gL,\gG)$
under the extra assumption \eqref{extra}. 
(This assumption is convenient below, but not essential.) 

We assume \eqref{extra} and define 
\begin{align}
\Delta_n = \sum_{k= A + \Lambda }^{B + \Gamma-1 } \XV_k - \sum_{k=A- \Lambda  }^{B- \Gamma-1 } \YW_k.
\end{align}

From Theorem \ref{newj2}, 
$BINGO(A + \Lambda, A - \Lambda; B + \Gamma, B - \Gamma)$ occurs if either
$\{ 0\le - \Delta_n <\XV_{B+\Gamma  } \}$ or if 
$\{ 0\le \Delta_n < \YW_{B-\Gamma } \}$. 
  We analyze $\P[ 0 \le \Delta_n < \YW_{B- \Gamma }]$. 

We compute first the variance of $\Delta_n$ and observe that,
using \eqref{extra}, 
\begin{align}\label{bba}
\sigma_n^2 
&:= \mathrm{Var}(\Delta_n) 
= \sum_{k= A + \Lambda }^{B + \Gamma-1 } \Var \XV_k + \sum_{k=A- \Lambda}^{B- \Gamma-1 } 
\Var \YW_k
\nonumber\\&
=\sum_{k= A + \Lambda }^{B + \Gamma-1 } k^{-2\ah} + \sum_{k=A- \Lambda}^{B- \Gamma-1 } k^{-2\ah}
\sim \frac{2}{2 \alpha -1}  
\Bigpar{A^{-(2 \alpha -1)}-B^{-(2\ga-1)}}
\nonumber\\&
\sim \frac{2}{2 \alpha -1}  A^{-(2 \alpha -1)}
\bigpar{1-\zeta^{2\ga-1}}.
\end{align}

We continue with the characteristic function 
\begin{align}
\phi_n(t) := \E\bigl[ e^{\ii t \Delta_n}\bigr]
= \prod_{k= A + \Lambda }^{B + \Gamma-1 }  \frac{k^{\ah}}{k^{\ah}-\ii t}
 \prod_{k=A- \Lambda}^{B- \Gamma-1 } \frac{k^{\ah}}{k^{\ah}+\ii t}. 
\end{align}
This is no longer real, but we can still estimate its absolute value as in
\eqref{byd} and \eqref{byf}, with minor modifcations,
and obtain \eqref{byff}.
Furthermore,
\begin{align}
\ln\phi_n(t)
&
= -\sum_{k= A + \Lambda }^{B + \Gamma-1 }\ln \Bigl(1-\frac{\ii t}{k^{\ah}}\Bigr)
- \sum_{k=A- \Lambda}^{B- \Gamma-1 }\ln \Bigl(1+\frac{\ii t}{k^{\ah}}\Bigr).
\end{align}
We consider $t=v/\sigma_n$ for a fixed real $v$ and obtain by 
Taylor expansions, recalling \eqref{bya} (with a trivial modification),
\eqref{bba} and \eqref{extra},
\begin{align}
\ln\phi_n\Bigpar{\frac{v}{\sigma_n}}
&
=\Bigl( \sum_{k= A + \Lambda }^{B + \Gamma-1 } - \sum_{k=A- \Lambda}^{B- \Gamma-1 }\Bigr)
\frac{\ii v/\sigma_n}{k^{\ah}}
- 
\Bigl( \sum_{k= A + \Lambda }^{B + \Gamma-1 } + \sum_{k=A- \Lambda}^{B- \Gamma-1 }\Bigr)
\frac{ v^2/\sigma_n^2}{2k^{2\ah}}\bigpar{1+o(1)}
\nonumber\\&
=-2\Lam\frac{\ii v/\sigma_n}{A^{\ah}}+2\Gam\frac{\ii v/\sigma_n}{B^{\ah}}
-\frac{v^2}{2}+o(1)
\nonumber\\&
\to 
\sqrt{\frac{2(2\ah-1)}{1-\zeta^{2\ga-1}}}\bigpar{-\lambda+\zeta^{\ga-1/2}\gam} v\ii 
 -\frac{v^2}{2}.
\label{bbf}
\end{align}
It follows by Fourier inversion and dominated convergence,
using \eqref{bbf} and \eqref{byff}, that, for any sequence
$x_n=o(\sigma_n)$,
with $c_1:=\sqrt{2(2\ah-1)/(1-\zeta^{2\ga-1})}$,
\begin{align}
\sigma_n f_n(x_n) &
= \frac{\sigma_n}{2\pi} \int_{-\infty}^{\infty} e^{-\ii x_n t}\phi_n (t)\dd t
= \frac{1}{2\pi} \int_{-\infty}^{\infty} e^{-\ii v x_n/\sigma_n}\phi_n
\Big(\frac{v}{\sigma_n}  \Big) \dd v 
\nonumber\\&
\to
\frac{1}{2\pi} \int_{-\infty}^{\infty} e^{-c_1(\lambda-\zeta^{\ga-1/2}\gam) v\ii-v^2/2} \dd v 
\nonumber\\&
= \frac{1}{\sqrt{2\pi}}  e^{-c_1^2(\lambda-\zeta^{\ga-1/2}\gam)^2/2}
.
\label{bbx}
\end{align}
Furthermore, using \eqref{byff} again, 
we have the uniform bound, for all real $x$,
\begin{align}\label{bbxx}
\sigma_n f_n(x) &
\le 
\frac{1}{2\pi} \int_{-\infty}^\infty\Bigabs{\phi_n\Bigpar{\frac{v}{\sigma_n} }} 
 \dd v 
\le \frac{1}{2\pi} \int_{-\infty}^{\infty} \frac{C}{1+v^2} \dd v 
\le C
.
\end{align}
We complete the proof as in \eqref{byi}--\eqref{byk}, 
now obtaining
\begin{align}
\P[ 0\le  \Delta_n < \YW_{B-\gG}]  
\sim
 \frac{1}{\sqrt{2\pi}}\sigma_n^{-1}B^{-\ah}
  e^{-c_1^2(\lambda-\zeta^{\ga-1/2}\gam)^2/2}.
\end{align}
 $\P[ 0< - \Delta_n < \XV_{B+ \Gamma  }]$ is similar, and  thus
\begin{align}
p(A,B,\gL,\gG)
\sim
 \frac{2}{\sqrt{2\pi}}\sigma_n^{-1}B^{-\ah}
  e^{-c_1^2(\lambda-\zeta^{\ga-1/2}\gam)^2/2}.
\label{bykk}
\end{align}
A simple calculation, using \eqref{extra} in \eqref{near_multiple}, shows
that the same asymptotics holds for $q(A,B,\gL,\gG)$, and thus
$p(A,B,\gL,\gG)\sim q(A,B,\gL,\gG)$. As explained at the beginning of the
proof, this implies the theorem.
\end{proof}
We end this section with two less precise estimates that are useful because
they do not require the
condition $B\ge\cg A$ in \refT{Tbingo}.

\begin{lemma}\label{LQ}
  Suppose that $0<A<B$, and that $\gL,\gG\ge0$ with $\gL< A$,
$B-\gG\ge A$,
  and\/ $\gG/B\le \frac{1}{8}\gL/A$. Then, for some constant $c>0$
  depending on $\ga$ only,
  \begin{equation}\label{lq}
    \Pr\Bigsqpar{\bigcup_{\ell\le\gG}BINGO(A+\gL,A-\gL;B+\ell,B-\ell)}
\le e^{-c\gL^2/A}.
  \end{equation}
\end{lemma}

\begin{proof}
Denote the event on   the left-hand side of \eqref{lq} by $\cE$.
We may assume $B+\gG\ge A+\gL$, since otherwise $\cE$ is empty.
By an argument similar to the proof of \refT{newj1},
the event $\cE$ occurs if and only if 
$\sum_{i=A+\gL}^{B+\gG}\XV_i \ge \sum_{j=A-\gL}^{B-\gG-1}\YW_j$.
Note that this event is monotonely decreasing in $\gL$; hence it suffices to
prove \eqref{lq} for $\gL\le A/2$ (and $\gG/B\le\frac{1}4\gL/A$), since we
otherwise may decrease $\gL$ to $A/2$ (changing $c$); we make these assumptions.

By Markov's inequality and independence, for every $t\ge0$,
\begin{align}
  \Pr[\cE]
&
=\Pr\lrsqpar{\sum_{i=A+\gL}^{B+\gG}\XV_i - \sum_{j=A-\gL}^{B-\gG-1}\YW_j\ge0}
\le \E e^{t\sum_{i=A+\gL}^{B+\gG}\XV_i - t\sum_{j=A-\gL}^{B-\gG-1}\YW_j}
\nonumber\\&
=\prod_{i=A+\gL}^{B+\gG}\E e^{t \XV_i} \prod_{j=A-\gL}^{B-\gG-1}\E e^{-t\YW_j}.
\label{ql}
\end{align}
Furthermore, when $-\infty < t < i^\ga$,
\begin{align}
  \E e^{t \XV_i}=\E e^{t \YW_i} = \frac{1}{1-ti^{-\ga}}.
\end{align}
Consequently, \eqref{ql} implies, for $0\le t\le \frac12A^{-\ga}$ and some
constant $C\ge1$ (depending on $\ga$), using the convexity of $j\mapsto
j^{-\ga}$ in the fourth inequality,
\begin{align}
\ln  \Pr[\cE]
&
\le-\sum_{i=A+\gL}^{B+\gG}\ln(1-ti^{-\ga})
-\sum_{j=A-\gL}^{B-\gG-1}\ln(1+tj^{-\ga})
\nonumber\\&
\le\sum_{i=A+\gL}^{B+\gG}\bigpar{ti^{-\ga}+t^2i^{-2\ga}} 
-\sum_{j=A-\gL}^{B-\gG-1}\bigpar{tj^{-\ga}-t^2j^{-2\ga}} 
\nonumber\\&
\le 
-\sum_{j=A-\gL}^{A+\gL-1} tj^{-\ga}
+\sum_{i=B-\gG}^{B+\gG} ti^{-\ga}
+2\sum_{i=A-\gL}^{\infty}t^2i^{-2\ga}
\nonumber\\&
\le-2\gL t A^{-\ga} + (2\gG+1)t(B-\gG)^{-\ga}+C t^2 A^{1-2\ga}
\nonumber\\&
\le-t\gL A^{-\ga} +C t^2 A^{1-2\ga},
\label{qm}
\end{align}
where the last inequality follows because the assumptions imply
$(2\gG+1)/(B-\gG) \le \gL/A$.

Now choose $t:=(2C)\qw\gL A^{\ga-1}$; then \eqref{qm} yields \eqref{lq},
with $c=1/4C$.
\end{proof}

\begin{lemma}  \label{LD}
For any $A<n$ and $\gL$ with $|\gL|\le A-1$, for some universal constants
$C,C'$. 
\begin{align}
  \Pr\bigsqpar{BINGO(A+\gL,A-\gL;n,n)}
\le \frac{C}{\sqrt{n-A}}e^{-\gL^2/(n-A)}
\le \frac{C'}{\gL}e^{-\gL^2/2(n-A)}.
\label{ld}
\end{align}
\end{lemma}

\begin{proof}
  Consider the Markov chain $(X_k,Y_k)_{2A}^\infty$, started at
$(X_{2A},Y_{2A})=(A+\gL,A-\gL)$.
We couple the chain with a simple random walk $(X^*_k,Y^*_k)_{2A}^\infty$,
also started at $(A+\gL,A-\gL)$, such that for every $k\ge 2A$,
\begin{equation}\label{dx}
  |X_k-Y_k|\ge |X_k^*-Y_k^*|;
\end{equation}
this can be achieved as follows.
If strict inequality holds in \eqref{dx}, so 
$  |X_k-Y_k|\ge |X_k^*-Y_k^*|+2$ since both sides have the same parity,
we may couple the next steps for the two chains arbitrarily. 
The same holds if 
$|X_k-Y_k|=|X_k^*-Y_k^*|=0$.
Finally, if $|X_k-Y_k|=|X_k^*-Y_k^*|>0$, we have to couple such that
if $|X^*_{k+1}-Y^*_{k+1}|=|X_k^*-Y_k^*|+1$, then
$|X_{k+1}-Y_{k+1}|=|X_k-Y_k|+1$; this is always possible, since the first
event has probability $1/2$, and the second has probability 
$\max\set{X_k^\ga,Y_k^\ga}/(X_k^{\ga}+Y_k^\ga)>1/2$.

Using this coupling, \eqref{dx} shows 
$  |X_{2n}-Y_{2n}|\ge |X_{2n}^*-Y_{2n}^*|$, and thus
\begin{align}
&  \Pr\bigsqpar{BINGO(A+\gL,A-\gL;n,n)}
=
  \Pr\bigsqpar{X_{2n}=Y_{2n}=n}
\nonumber\\&\qquad
\le   \Pr\bigsqpar{X^*_{2n}=Y^*_{2n}}
=
  \Pr\bigsqpar{\Bin(2n-2A,1/2)=n-A-\gL}
\nonumber\\&\qquad
=2^{-(2n-2A)}\binom{2n-2A}{n-A-\gL}
\end{align}
and \eqref{ld} follows by standard calculations using Stirling's formula.
\end{proof}

\section{A Basic Case}
\label{sec:basic}

Here we prove a modified version of Theorem \ref{t1}.  
Initially $\xi_1,\ldots,\xi_{2n-2}=\pm 1$ are uniform and independent.
We set 
\beq\label{defA} A = \lfloor\ln^{10}n\rfloor \eeq
We shall be splitting the walk into an initial part, until the coordinates
sum to $2A$, and a main part, until the coordinates sum to $2n$.  See
remark \ref{r1} for further comments.

We {\em condition} on $S_{2A-2}=0$ {\em and} $S_{2n-2}=0$.  
We may and shall consider the
$S_i$ in two regimes.  For $0\leq i\leq 2A-2$ the $S_i$ form a random
excursion, beginning ($i=0$) and ending ($i=2A-2$) at zero. 
For $2A-2 \leq i \leq 2n-2$ the $S_i$
form a random excursion beginning ($i=2A-2$) and ending ($i=2n-2$) at zero.
In the state space the walk begins at $(1,1)$, goes to $(A,A)$ and then goes
to $(n,n)$.  We note, importantly, that the two sides of the walk are
mutually independent excursions.  Let {\bf COND} denote this condition.
The function $L=L_{2n-2}$
splits naturally into two parts:
\begin{align}
\label{Linit} L^{init}& = \sum_{i=1}^{2A-2} \frac{S_i^2}{i^2},
\\
\label{modL} L^{main} &= \sum_{i=2A-1}^{2n-2} \frac{S_i^2}{i^2}.
\end{align}

\begin{theorem}\label{t4}  Under {\bf COND},
for $c \in (0,1)$,
\beq\label{6a} \Pr[L^{main} \leq c\ln n\mid\COND] = e^{-(K(c)+o(1))\ln n}  \eeq
with (as in Theorem \ref{t1})
\beq\label{7a} K(c) =\frac{(1-c)^2}{8c}  
\eeq
\end{theorem}
\begin{remark}\label{r1}  There is considerable flexibility in the choice of the
breakpoint $A$.  The basic object is to protect against rare events.  Our
basic argument will breaks down when, say, $|S_i|\geq 0.01i$.  This occurs
with probability exponentially small in $i$.  However, Theorem \ref{t1}
deals with polynomially small (in $n$) probabilities.  Restricting to
$i\geq A$, exponentially small in $i$ is less than polynomially small
in $n$ and hence negligible.  The split at $A$ should be considered an
artifact of the proof and it is quite possible that an argument exists
that does not use this artifical split.
\end{remark}

Both the restriction to $i\geq A$ and the restriction to a random
excursion shall be later removed.

We shall actually find the asymptotics of the Laplace transform of $L^{main}$.
For notational convenience, given $\ah > 1$ we define
\beq\label{defl} \lam = \frac{\ah(\ah-1)}{2}  \eeq

\begin{theorem}\label{ta1}  For any $\ah > 1$
\beq\label{c1}
\E[e^{-\lam L^{main}}\mid\COND] = n^{-(\ah-1)/2+o(1)} \eeq
\end{theorem}

As $\ah$ ranges over $(1,\infty)$, $t:=-\lam$ ranges over the negative 
reals.  Theorem \ref{ta1} then gives the asymptotics of the Laplace transform
of $L^{main}\mid\COND$:
letting
$\WL_n$ be $L^{main}\mid\COND$ for a particular value of $n$,
\begin{equation}\label{LDP0new}
  \lim_\ntoo\frac{1}{\ln n} \ln \E e^{t \WL_n} = \Lam(t):= -(\ah-1)/2,
\qquad t<0.
\end{equation}
Then (as done in more detail in \S \ref{sec:new_proof}), 
the Legendre transform of $\Lam(t)$ is, by a simple calculation,
\begin{equation}\label{gLxnew}
  \gLx(x)=K(x)
\end{equation}
and the
G\"artner--Ellis theorem 
\cite[Theorem 2.3.6]{DZ}
yields the asymptotics of $\Pr[\WL_n\leq c\ln n]$ of Theorem \ref{t4}.

\begin{remark}\label{whyLaplace}  
The main contribution to 
$\E[e^{-\lam L^{main}}\mid\COND]$ 
comes when $L^{main}\approx c \ln n$ with $c=(2\ah -1)^{-1}$, i.e., $\ah =
\frac{c+1}{2c}$. 
\end{remark}

Now we study $\Pr[BINGO(A,A;n,n)]$ as the sum of the probabilities of all paths {\bf P}
from $(A,A)$ to $(n,n)$.  Let $P_{2A}=(A,A),\ldots,P_{2n}=(n,n)$ denote the points of path
{\bf P}, $P_i$ having some of coordinates $i$, $2A\leq i \leq 2n$.  Let $P_i=(x_i,y_i)$.
Critically, we parametrize, as in \eqref{XYgD},
\beq\label{8a}  x_i=\frac{i+\gd_i}{2} \mbox{ so that } y_i= \frac{i-\gd_i}{2}  \eeq
Here $\gd_i$ reflects the ``distance" of the path from the main diagonal.
By $\Pr({\bf P})$ we mean the probability of following precisely the path {\bf P}.
The $\Pr({\bf P})$ vary in an interesting way.  The numerators multiply out the same with
factors $i^{\ah}$ for $A\leq i < n$ and $j^{\ah}$ for $A\leq j < n$.  The denominator factor
$x_i^{\ah}+y_i^{\ah}$ is minimal when $x_i=y_i=\frac{i}{2}$.  We define
\beq\label{9a} FIT_i = \frac{2(i/2)^{\ah}}{x_i^{\ah} + y_i^{\ah}}  \eeq
Then $FIT_i= f(\eps)$ where $\eps= \gd_i/i$ and
\beq\label{10a}  f(\eps) = \frac{2}{(1+\eps)^{\ah}+(1-\eps)^{\ah}}  
\le1
\eeq
We shall make critical use of the asymptotics
\beq\label{10.5a} \ln(f(\eps)) \sim -\lam\eps^2 \mbox{ as } \eps\ra 0 \eeq
Set 
\beq\label{11a} FIT= FIT({\bf P}) = \prod_{i=2A}^{2n-1} FIT_i  \eeq
Each $FIT_i\leq 1$ and hence $FIT\leq 1$.  A low $FIT$ tells us that the path {\bf P} is
relatively unlikely.  Roughly, paths {\bf P} which stay close to the main diagonal will have
a high $FIT$, meaning they will be more likely than those the stray far from the main diagonal.
We now split
\beq\label{12b} \Pr[{\bf P}] = BASE\cdot FIT\eeq
Here $BASE$ is what {\bf P} would be if the terms $x_i^{\ah}+y_i^{\ah}$ were replaced by $2(i/2)^{\ah}$
and $FIT$ is the additional factor with the actual $x_i,y_i$, $2A\leq i < 2n$.
Then the denominator would be precisely the product of $2(i/2)^{\ah}$ over $2A\leq i < 2n$.  That is
\beq\label{13b}
BASE = \frac{\prod_{i=A}^{n-1}i^{\ah}\prod_{j=A}^{n-1}j^{\ah}} {\prod_{i=2A}^{2n-1} 2(i/2)^{\ah}} \eeq
$BINGO(A,A;n,n)$ is the sum of $BASE\cdot FIT({\bf P})$ over all 
$\binom{2(n-A)} {n-A}$ paths from
$(A,A)$ to $(n,n)$.  We rewrite (\ref{12b}) with the {\em exact} formula
\beq\label{15b} \Pr[BINGO(A,A;n,n)] = BASE\cdot 
\binom{2(n-A)}{n-A} \cdot \E[FIT({\bf P})]  
\eeq
where expectation is over a uniformly chosen path from $(A,A)$ to $(n,n)$.  Stirling's Formula
asymptotics 
give
\beq
BASE = n^{-\ah/2+o(1)}2^{-2(n-A)}
\eeq
and
\beq\label{16b} \Pr[BINGO(A,A;n,n)] 
= n^{-(1+\ah)/2+o(1)} \cdot \E[FIT({\bf P})]  
\eeq
Applying (\ref{7.2a}) we deduce 
\beq\label{17b}  \E[FIT({\bf P})] = n^{(1-\ah)/2+o(1)} . \eeq
\begin{remark}  We had originally hoped  to apply (\ref{16b}) in reverse.  That is, a
combinatorial (or other) argument for the asymptotics of $\E[FIT({\bf P})]$ would 
yield an alternate proof, a non-Book Proof, for $\Pr[BINGO]$.  It was surprising that
the continuous time approach led to (\ref{17b}), which is quite difficult to prove
directly.
\end{remark}

Now we try to estimate $FIT$ using (\ref{10.5a}).  The technical difficulty as that we do not
have $\eps = \gd_i/i = o(1)$ tautologically.  Call a walk {\bf P} {\em weird} if $|S_i|>i^{0.99}$
for some $A\leq i \leq n$.  Otherwise call {\bf P} {\em normal}.  Large deviation results give
that the probability {\bf P} is weird for a particular $i$ is at most $\exp[-i^{0.98}/2]$.  We
only look at $i\geq A$. We have selected $A$ so that this probability is subpolynomial.   As
$FIT({\bf P})\leq 1$ tautologically, the affect on $\E[FIT({\bf P})]$ of weird {\bf P}
is negligible.  Hence in calculating $\E[FIT({\bf P})]$ we can restrict ourselves to normal
{\bf P}.  Normal {\bf P} have $\eps = |S_i|/i < i^{-0.01} = o(1)$ uniformly.  We apply
(\ref{10.5a}), each $\ln(FIT_i)\sim -\lam S_i^2i^{-2}$ so that 
$\ln(FIT) \sim -\lam L^{main}$.  Therefore
\beq\label{18b}  \E[e^{-\lam L^{main}}\mid\COND] = n^{(1-\ah)/2 + o(1)} \eeq
as desired, giving Theorem \ref{ta1} and hence Theorem \ref{t4}.

We now extend Theorem \ref{t4} to $L=L^{init}+L^{main}$.  As $L^{init}\geq 0$,
\beq\label{19b} \Pr[L\leq c\ln n\xCOND] \leq \Pr[L^{main}\leq c\ln n\xCOND] 
\leq e^{-(K(c)+o(1))\ln n}  .\eeq
Now we show $L^{init}$ under {\bf COND} is appropriately negligible.  We have
$\xi_1,\ldots,\xi_{2A-2}=\pm 1$ conditioned on their sum being zero.
$S_i=\xi_1+\ldots+\xi_i$.  A standard second moment calculation gives
the precise value $E[S_i^2] = i - i(i-1)(2A-3)^{-1}$ but we shall
only use $E[S_i^2]\leq i$.  (That is, the conditioning lowers the
variance.)  Then, using (\ref{defA})
\beq\label{20b} \E[L^{init}\xCOND] \leq \sum_{i=1}^{2A-2} \frac{i}{i^2} \leq (10+o(1))\ln\ln n .\eeq
By Markov's Inequality, with $n$ sufficiently large,  $L^{init}\leq 21\ln\ln n$ with probability at least $0.5$.
We have created $L^{init},L^{main}$ to be independent so with probability at least
$0.5\cdot e^{-(K(c)+o(1))\ln n}$ both $L^{init}\leq 21\ln\ln n$ and $L^{main}\leq c\ln n$.
Hence
\beq\label{21b} \Pr[L\leq c\ln n + 21\ln\ln n\xCOND] \geq \frac{1}{2}e^{-(K(c)+o(1))\ln n} \eeq
The multiplicative factor of $\frac{1}{2}$ and the additive factor of $21\ln\ln n$ get
absorbed in the asymptotics, giving
\beq\label{22b} \Pr[L\leq c\ln n\xCOND] \geq e^{-(K(c)+o(1))\ln n}  \eeq

We have shown:
\begin{theorem}\label{e23} Under {\bf COND},
\beq\label{23b} \Pr[L\leq c\ln n\mid \COND] = e^{-(K(c)+o(1))\ln n} . \eeq
\end{theorem}

\section{The Lower Bound}
\label{sec:lower}

Let $|\Lam|\leq \sqrt{A}$, $|\Gam| \leq \sqrt{n}$.  We generalize
the Basic Case.
Initially $\xi_1,\ldots,\xi_{2n-2}=\pm 1$ are uniform and independent.
Here we {\em condition} on $S_{2A-2}=2\Lam$ {\em and} $S_{2n-2}=2\Gam$.
We may and shall consider the
$S_i$ in two regimes.  For $0\leq i\leq 2A-2$ the $S_i$ form a random
excursion, beginning ($i=0$) at 0  and ending ($i=2A-2$) at $2\Lam$. 
For $2A-2 \leq i \leq 2n-2$ the $S_i$
form a random excursion beginning ($i=2A-2$) at $2\gL$
and ending ($i=2n-2$) at $2\Gam$.
As in \refS{sec:basic}, the two excursions are independent.
Let {\bf COND}($\Lam$,$\Gam$) denote this condition.

\begin{theorem}\label{e23near} 
Uniformly in $|\Lam|\leq \sqrt{A}$, $|\Gam| \leq \sqrt{n}$,
under {\bf COND($\Lam$,$\Gam$)},
\beq\label{23bnear} 
\Pr[L\leq c\ln n\mid {\mathbf{COND}(\Lam,\Gam)}] 
= e^{-(K(c)+o(1))\ln n}  \eeq
\end{theorem}
\begin{proof}
  This follows the same lines as Theorem \ref{e23}.  
The critical preferential attachment Theorem \ref{bingoaann} and \eqref{7.2a}
are replaced by \refC{Cbingo0} and \eqref{nearrough}.
\end{proof}

From Theorem \ref{e23near} we derive the lower bound of Theorem \ref{t1}.

\begin{theorem}\label{t1lower}
For $c \in (0,1)$,
\beq\label{2alower} \Pr[L \leq c\ln n] \geq e^{-K(c)+o(1))\ln n}  \eeq
with
\beq\label{3ax} K(c) =\frac{(1-c)^2}{8c}  
\eeq
\end{theorem}

\begin{proof}
We split (sum over
$|\Lam|\leq \sqrt{A}$, $|\Gam| \leq \sqrt{n}$)  
\beq\label{near2}
\Pr[L\leq c\ln n] \geq \sum_{\Lam,\Gam}\Pr[L\leq c\ln n\mid\CONDLG]\Pr[\CONDLG]
\eeq

These conditionings are disjoint.  An unrestricted random walk has probability $\Omega(1)$ of
having these ``reasonable" values at $2A-2$ and $2n-2$,  so the sum of the probabilities of {\bf COND}
is $\Omega(1)$.  
From Theorem \ref{e23near} the conditional probabilities of $L\leq c\ln n$ are
all bounded from below.
\end{proof}

\section{The Upper Bound}
\label{sec:upper}

We employ coupling arguments to give upper bounds on the large deviation of $L$.

\begin{theorem}\label{firstcoupling}
For any $\Lam$ with $|\Lam|\leq A$ and any $z$
\beq\label{upper1}  \Pr[L^{main}\leq z\mid\COND (\Lam,0)] 
\leq \Pr[L^{main}\leq z\mid\COND (0,0)]  \eeq
\end{theorem}

\begin{proof}  
We couple paths $P_{2A}=(A+\Lam,A-\Lam)$,\ldots,$P_{2n}=(n,n)$ with paths $P_{2A}^*= (A,A)$,\ldots,$P_{2n}^*=(n,n)$.
Determine the random paths $P,P^*$ sequentially, starting at $2A$.
Let $t$ be the first value (if any)
for which, setting $P_t=(a,b)$,  either $P_t^*=(a,b)$ or $P_t^*=(b,a)$.
In the first case couple $P_s=P_s^*$ for all $t\leq s \leq 2n$.  In the second case couple $P_s^*$ to
be $P_s$ with coordinates reversed (that is, flip the path on the diagonal) for all $t\leq s\leq 2n$.
For any paired $P,P^*$, $|S_i|> |S_i^*|$ for $2A\leq i < t$ and $|S_i|=|S_i^*|$ for $t\leq i \leq 2n$.
Thus $L^{main}(P) \geq L^{main}(P^*)$ and (\ref{upper1}) follows.
\end{proof}

\begin{corollary}\label{couplingcor}
For any $\Lam$ with $|\Lam|\leq A$ and any $z$
\beq\label{upper2}  \Pr[L\leq z\mid\COND (\Lam,0)] \leq \Pr[L^{main}\leq z\mid\COND (0,0)]  \eeq
\end{corollary}

\begin{proof} $L^{main}\leq L$ so $\Pr[L\leq z] \leq \Pr[L^{main}\leq z]$.
\end{proof}

\begin{corollary}\label{couplingcor1}
For any $z$
\beq\label{upper3} \Pr[L\leq z\mid P_{2n}=(n,n)]  
\leq \Pr[L^{main}\leq z\mid\COND (0,0)]  \eeq
\end{corollary}

\begin{proof}  The event $P_{2n}=(n,n)$ is the disjoint disjunction of the events $\COND (\Lam,0)$.
As, from (\ref{upper2}), $\Pr[L\leq z]$ is uniformly bounded conditional under each of the events $\COND (\Lam,0)$,
it has the same bound conditional on their disjunction.
\end{proof}

\begin{theorem}\label{couplethm2}  
For any $\Gam$ with $|\Gam|\leq n$, 
\beq\label{upper4} \Pr[L\leq z\mid P_{2n}=(n+\Gam,n-\Gam)    ] \leq \Pr[L \leq z\mid P_{2n}=(n,n) ]
\eeq
\end{theorem}

\begin{proof}  We reverse time, and consider the random walk starting at $P_{2n}^*=(n+\Gam,n-\Gam)$
and ending at $P_2^*=(1,1)$.  That is, at a state $(a,b)$ one moves to either $(a-1,b)$ or $(a,b-1)$
with the probabilities that the random walk from $(1,1)$ to $(a,b)$ goes through those states.
We couple
walks $P_{2n}^*$,\ldots $P_2^*$ with walks $P_{2n}$,\ldots,$P_2$.  Let $t$ be the first value
(here, highest index value) so that, with $P_t^*=(a,b)$, either $P_t=(a,b)$ or $P_t=(b,a)$.  In
the first case we couple $P_s^*=P_s$ for $2\leq s\leq t$ and in the second case $P_s$ is $P_s^*$
with coordinates reversed for $2\leq s\leq t$.  For any paired paths $P^*,P$, $L(P^*)\geq L(P)$
and so the lower tail inequality (\ref{upper4}) follows.
\end{proof}

\begin{corollary}\label{couplingcor2}  
For any $\Gam$ with $|\Gam|\leq n$, 
\beq\label{upper5} \Pr[L\leq z\mid P_{2n}=(n+\Gam,n-\Gam)    ] 
\leq \Pr[L^{main}\leq z\mid \COND (0,0)]  \eeq
\end{corollary}

\begin{proof} Combine Corollary \ref{couplingcor1} and Theorem \ref{couplethm2}.
\end{proof}

\begin{theorem}\label{couplingthm3} Let $\xi_3,\ldots,\xi_{2n}=\pm 1$ independently and uniformly. Let
$S_t$ be the walk with initial value $S_2=0$ and step $S_t=S_{t-1}+\xi_t$.
Set $L = \sum_{i=2}^{2n} S_i^2/i^2$.  Then
\beq\label{upper6} \Pr[L\leq z] 
\leq \Pr[L^{main}\leq z\mid \COND (0,0)]  \eeq
\end{theorem}

\begin{proof}  The unrestricted walk is the disjoint disjunction of the excursions
ending at $P_{2n}=(n+\Gam,n-\Gam)$.  Corollary \ref{couplingcor2} gives the upper
bound under any of these conditions, so the upper bound holds under their disjunction.
\end{proof}

We set $z=c\ln n$.  Theorems \ref{couplingthm3} and \ref{t4}  
yield the upper bound to Theorem \ref{t1} and hence,
together with Theorem \ref{t1lower},
prove Theorem \ref{t1}.

\section{Brownian Approximations}\label{sec:brownianapprox}

In this Section, we introduce a Brownian Analogue for $L$, and establish that for the purposes of establishing Theorem \ref{t1}, it is enough to establish the corresponding statement for the Brownian analogue. To this end, let 
$\{B_t : t \geq 0\}$ be a standard Brownian motion (with $B_0=0$). Thus $B_n$ is a natural approximation of $S_n$.

Recall $L=L_n$ from \eqref{1a} and define the two natural approximations
%
\begin{align}
\tL&=\tL_n = \sumin \frac{B_i^2}{i^2}, \label{tL}
\\
\hL&=\hL_n = \int_1^n \frac{B_t^2}{t^2}\dd t.\label{hL}
\end{align}

\noindent
We introduce a cutoff $A$; $A:=\floor{\ln^{10}n}$ 
as in \eqref{defA} works in this case as well, except that we assume that
$A$ is an even integer (this is convenient and simplifies the argument in
\refL{L3} below, but is not essential).  
Define 
\begin{align}
  L' &= L'_{n} = \sum_{i=A+1}^n \frac{S_i^2}{i^2}
\\
  \tL' &= \tL'_{n} = \sum_{i=A+1}^n \frac{B_i^2}{i^2}
\\
  \hL' &= \hL'_{n} = \int_A^n \frac{B_t^2}{t^2}\dd t
\intertext{and}
  L'' &= L''_{n} = \sum_{i=A+1}^n \frac{(S_i-S_A)^2}{i^2} \label{L''}
\\
  \tL'' &= \tL''_{n} = \sum_{i=A+1}^n \frac{(B_i-B_A)^2}{i^2}\label{tL''}
\\
  \hL'' &= \hL''_{n} =  \int_A^n \frac{(B_t-B_A)^2}{t^2}\dd t \label{hL''}
\end{align}

Note that
\begin{equation}\label{EL}
  \E L_n = \E L_n' = \sumin \frac{1}{i} = \ln n+O(1)
\end{equation}
and
\begin{equation}\label{EL''}
  \E L''_n = \int_1^n \frac{1}{t}\dd t = \ln n.
\end{equation}

\noindent
Throughout this discussion, 
$C$ denotes some unspecified finite constants, changing from one occurrence
to the next.
(In contrast to $c$, which is our main parameter.) We implicitly assume that $n$ is large. At least, assume $n\ge 8$ throughout, so
$\ln\ln n\ge1$.


Lemmas \ref{L1}--\ref{L5} establish that the random variable $L_n$, and those
defined in \eqref{tL}--\eqref{hL''} are equivalent for our purposes.

\begin{lemma}\label{L1}
For any $c>0$ and $\eps>0$, for $n$ large enough,
\begin{align}
    \P\bigpar{L_n\le c\ln n}&\ge \frac12\P\bigpar{L_n''\le (c-\eps)\ln n},
\label{l1}  \\
    \P\bigpar{\tL_n\le c\ln n}&\ge \frac12\P\bigpar{\tL_n''\le (c-\eps)\ln n},
\label{l1t}\\
    \P\bigpar{\hL_n\le c\ln n}&\ge \frac12\P\bigpar{\hL_n''\le (c-\eps)\ln n}.
\label{l1h}
\end{align}
  \end{lemma}

\begin{proof}
The proof of all three parts are identical, up to obvious (notational)
changes.
Hence we consider only \eqref{l1}.

  By Minkowski's inequality (the triangle inequality in $\ell^2$),
\begin{equation}
\sqrt{L_n'} \le \sqrt{L_n''} + \biggpar{\sum_{i=A+1}^n \frac{S_A^2}{i^2}}\qq    
\le \sqrt{L_n''} + \frac{|S_A|}{\sqrt{A}}
  \end{equation}
and thus
\begin{equation}\label{xa}
\sqrt{L_n}= \sqrt{L_n'+L_A}
\le \sqrt{L_n'} + \sqrt{L_A}
\le \sqrt{L_n''} + \frac{|S_A|}{\sqrt{A}}  + \sqrt{L_A}.
  \end{equation}
Furthermore, $\E S_A^2=A$ and
$\E L_A = O\bigpar{\ln A}=O\bigpar{\ln\ln n}$ by \eqref{EL}; hence,
by Chebyshev's and Markov's inequalities, for a suitable $C$,
\begin{align}
\P\Bigpar{\frac{|S_A|}{\sqrt{A}}> C} &\le \frac{1}4,
\\
\P\Bigpar{{L_A}> C \ln\ln n} &\le \frac{1}4,
\end{align}
and thus
\begin{equation}\label{xb}
  \begin{split}
\P\Bigpar{\frac{|S_A|}{\sqrt{A}}+\sqrt{L_A}> C\sqrt{\ln\ln n}} 
\le \frac{1}2.
  \end{split}
\end{equation}
Since $L_n''$ is independent of $S_A$ and $L_A$, it follows from \eqref{xa}
and \eqref{xb} that
\begin{align*}
    \P\bigpar{L_n\le c\ln n}
&\ge \P\bigpar{L_n''\le (c-\eps)\ln n}
\P\Bigpar{\frac{|S_A|}{\sqrt{A}}+\sqrt{L_A}\le C\sqrt{\ln\ln n}} 
\\&
\ge \frac12\P\bigpar{L_n''\le (c-\eps)\ln n}.
\qedhere
\end{align*}
\end{proof}

Obviously, $L_n\ge L_n'$, $\tL_n\ge\tL_n'$ and $\hL_n\ge\hL_n'$.
The next lemma says that $L_n'$ is stochastically larger than $L_n''$, and so on.
\begin{lemma}
  \label{L3}
For any $y\ge0$,
\begin{align}
\P\bigpar{L_n \le y} &\le  \P\bigpar{L_n'\le y} \le \P\bigpar{L_n''\le y},\label{l3}
\\ 
\P\bigpar{\tL_n \le y} &\le  \P\bigpar{\tL_n'\le y} \le \P\bigpar{\tL_n''\le y},
\label{l3t}\\ 
\P\bigpar{\hL_n \le y} &\le  \P\bigpar{\hL_n'\le y} \le \P\bigpar{\hL_n''\le y}.
\label{l3h}
\end{align}
\end{lemma}
\begin{proof}
Consider first \eqref{l3}.
  Define $\bS_i:=S_i-S_A$ for $i\ge A$.
Then $\bS_i$, $i\ge A$, is a simple random walk starting at $\bS_A=0$.

If we condition $(S_i)_{i\ge A}$ on $S_A=x$, we obtain a simple random walk
starting at $x$. This has the same distribution as $x+\bS_i$, but we shall
use a different coupling defined as follows.
Recall that $A$ is chosen to be even, and thus $S_A$ is an even integer.

For a given even integer $x$, define
the stopping time $\tau:=\inf\set{k\ge A:\bS_k=x/2}$, and 
\begin{equation}
  \bSx_i:=
  \begin{cases}
    x-\bS_i,& A\le i\le \tau,
\\
\bS_i, & i>\tau.
  \end{cases}
\end{equation}
Then $\bSx_i$ is a simple random walk, started at $\bSx_A=x$, and thus
$(\bSx_i)_A^\infty$ 
has the same distribution as
$(x+\bS_i)_A^\infty$. 
Furthermore, it is easily seen that, for all $i\ge A$, 
\begin{equation}
  |\bSx_i|\ge |\bS_i|.
\end{equation}
(To see this, we may  by symmetry assume $x\ge0$.
It suffices to consider $A\le i\le \tau$, and then $\bS_i\le x/2$, and thus
either 
$\bS_i\le 0$ and $\bSx_i=x+|\bS_i|$, or $0<\bS_i\le x/2 \le \bSx_i$.)

Consequently, for every even integer $x$ and every $y\ge0$,
\begin{equation}
  \begin{split}
\P\lrpar{L''_n \le y}
&=
\P\lrpar{\sum_{i=A+1}^n \frac{\bS_i^2}{i^2}  \le y}  
\ge
\P\lrpar{\sum_{i=A+1}^n \frac{\bigpar{\bSx_i}^2}{i^2}  \le y}  
\\&
=
\P\lrpar{\sum_{i=A+1}^n \frac{S_i^2}{i^2}  \le y \,\Big|\, S_A=x}  
=
\P\bigpar{L'_n \le y \mid S_A=x}.  
  \end{split}
\end{equation}
Thus, $\P\bigpar{L''_n \le y}\ge \P\bigpar{L'_n \le y\mid S_A}$, and thus
we obtain \eqref{l3} by taking the expectation.

The proofs of \eqref{l3t} and \eqref{l3h} are the same, with $S_n$ replaced
by $B_t$.
\end{proof}

\begin{lemma}\label{L4}
  For every $\eps>0$, $c>0$ and $a<\infty$,
  \begin{align}
\P(L_n''\le c\ln n) &\le \P\bigpar{\tL_n''\le (c+\eps)\ln n} + O\bigpar{n^{-a}},
\label{l4a}\\
\P(\tL_n''\le c\ln n) &\le \P\bigpar{L_n''\le (c+\eps)\ln n} + O\bigpar{n^{-a}},
\label{l4b}\\
\P(\tL_n''\le c\ln n) &\le \P\bigpar{\hL_n''\le (c+\eps)\ln n} + O\bigpar{n^{-a}},
\label{l4c}\\
\P(\hL_n''\le c\ln n) &\le \P\bigpar{\tL_n''\le (c+\eps)\ln n} + O\bigpar{n^{-a}},
\label{l4d}
  \end{align}
\end{lemma}

\begin{proof}
By \cite{KMT}, 
  there exists a coupling (the 'dyadic coupling') of the simple random walk
$(S_i)_{i\ge0}$
and the Brownian motion $(B_t)_{t\ge0}$ 
such that with probability $1-O(n^{-a})$,
for some constant $C_a$,
\begin{equation}\label{kmt}
  \max_{i\le n}|S_i-B_i|\le C_a\ln n,
\end{equation}
see also \cite[Chapter 7]{LL}.
If \eqref{kmt} holds, then 
$|(S_i-S_A)-(B_i-B_A)|\le 2C_a\ln n$
for $A\le i\le n$,
and thus, by Minkowski's inequality,
\begin{equation}\label{sw}
  \begin{split}
|  \sqrt{L_n''}-\sqrt{\tL_n''}|
&\le 
\lrpar{\sum_{i=A+1}^n \frac{\bigpar{(S_i-S_A)-(B_i-B_A)}^2}{i^2}}\qq
\\&
\le 2C_a\frac{\ln n}{\sqrt A}
=o(1).
  \end{split}
\end{equation}
Hence, \eqref{l4a} and \eqref{l4b} follow.

In order to prove \eqref{l4c}--\eqref{l4d}, we introduce yet another version
of $L_n$:
\begin{equation}
\label{chL''}
\chL''=\chL''_{n}
 :=  \int_A^n \frac{(B_{\ceil{t}}-B_A)^2}{t^2}\dd t 
=\sum_{i=A+1}^n \frac{(B_{i}-B_A)^2}{i(i-1)}.
\end{equation}
Then, see \eqref{tL''},
\begin{equation}\label{nw}
  \tL_n''\le\chL_n''\le \frac{A+1}{A}\tL_n''
=\bigpar{1+o(1)}\tL_n''.
\end{equation}
Moreover, by simple standard properties of Brownian motion,
\begin{align}
&\P\bigpar{\sup_{t\le n} |B_{\ceil t}-B_t|>\ln n}    
\le n\P\bigpar{\sup_{0\le t\le 1} |B_{1}-B_t|>\ln n}  
\notag
\\&\qquad
\le n\P\bigpar{\sup_{0\le t\le 1} |B_t|>\tfrac12\ln n}    
\le 4n\P\bigpar{B_1>\tfrac12\ln n}   
\le C n e^{-\ln^2n/8}
\notag
\\&\qquad
=O\bigpar{n^{-a}}. 
\end{align}
Hence, similarly to \eqref{sw},
with probability $1-O\bigpar{n^{-a}}$,
\begin{equation}\label{sw2}
  \begin{split}
|  \sqrt{\chL_n''}-\sqrt{\hL_n''}|
&\le 
\frac{\ln n}{\sqrt A}
=o(1).
  \end{split}
\end{equation}
We obtain \eqref{l4c} and \eqref{l4d} from \eqref{nw} and \eqref{sw2}.
\end{proof}

\begin{lemma}\label{L5}
  For every $\eps>0$, $c>0$ and $a<\infty$,
  \begin{align}
\P(L_n\le c\ln n) &\le2 \P\bigpar{\hL_n\le (c+\eps)\ln n} + O\bigpar{n^{-a}},
\label{l5a}\\
\P(\hL_n\le c\ln n) &\le2 \P\bigpar{L_n\le (c+\eps)\ln n} + O\bigpar{n^{-a}}.
\label{l5b}
  \end{align}
\end{lemma}

\begin{proof}
  By  \eqref{l3}, \eqref{l4a}, \eqref{l4c} and \eqref{l1h},
  \begin{align}
\P(L_n\le c\ln n) 
&\le
\P(L_n''\le c\ln n) 
\le \P\bigpar{\tL_n''\le (c+\eps)\ln n} + O\bigpar{n^{-a}}
\nonumber\\&    
\le \P\bigpar{\hL_n''\le (c+2\eps)\ln n} + O\bigpar{n^{-a}}
\nonumber\\&    
\le 2\P\bigpar{\hL_n\le (c+3\eps)\ln n} + O\bigpar{n^{-a}},  
  \end{align}
which yields \eqref{l5a} after replacing $\eps$ by $\eps/3$.

Similarly, \eqref{l3h}, \eqref{l4d}, \eqref{l4b} and \eqref{l1} yield
  \begin{align}
\P(\hL_n\le c\ln n) 
&\le
\P(\hL_n''\le c\ln n) 
\le \P\bigpar{\tL_n''\le (c+\eps)\ln n} + O\bigpar{n^{-a}}
\nonumber\\&   
\le \P\bigpar{L_n''\le (c+2\eps)\ln n} + O\bigpar{n^{-a}}
\nonumber\\&
\le 2\P\bigpar{L_n\le (c+3\eps)\ln n} + O\bigpar{n^{-a}}.
  \end{align}
\end{proof}

Consequently, it does not matter whether we use $L_n$ or $\hL_n$ (or $\tL_n$) in
Theorem \ref{t1}; the different versions are equivalent.

\section{Analysis of the Brownian versions}
\label{sec:brownian_analysis}

Note from \eqref{tL}--\eqref{hL} that both $\tL_n$ and $\hL_n$ are quadratic functionals of Gaussian variables.
There is a general theory for such studying large deviation for such variables.
This facilitates a direct analysis of the moment generating function of
\eqref{hL}.

\subsection{Moment generating function of $\hL$}
We utilize the general theory of Gaussian Hilbert Spaces to compute the moment generating function of $\hL_n$. For the convenience of the reader, we include a brief summary, relevant for this application, in \refApp{App}, and refer the interested reader to \cite[Chapters VII and VI]{SJIII} for further details. 

By \refT{T+} and \refL{LhL}, for every $t<(2\max\gl_j)\qw$, 
\begin{equation}\label{mgfhL}
  \E e^{t\hL_n}=\prod_j \bigpar{1-2\gl_j t}\qqw,
\end{equation}
where $(\gl_j)$ are the non-zero eigenvalues of the integral operator
\begin{equation}\label{T}
  \begin{split}
  Tf(x)&:=\inton \Bigpar{\frac{1}{1\lor x\lor y}-\frac{1}{n}} f(y)\dd y,
  \end{split}
\end{equation}
acting in $L^2\ono$.
As shown in \refApp{App}, 
see Remark \ref{Rcompact},
$T=T_n$ is a positive compact operator, and
thus $\gl_j>0$; furthermore, $\sum_j\gl_j=\E\hL_n=\ln n<\infty$.

Suppose that $f$ is an eigenfunction with a non-zero eigenvalue $\gl$. Thus
$f\in L^2\ono$ is not identically 0, and $Tf=\gl f$.
It follows from \eqref{T} by dominated convergence that
$Tf(x)$ is continuous in $x\in\on$; 
thus $f=\gl\qw Tf$ is continuous on $\on$.
Similarly, $f=\gl\qw Tf$ is  constant on $\oi$, and $f(n)=0$.
By \eqref{T}, we have
\begin{equation}
\gl f(x)=Tf(x)=
\int_0^{1\lor x} \Bigpar{\frac{1}{1\lor x}-\frac{1}{n}} f(y)\dd y
+ \int_{1\lor x}^n \Bigpar{\frac{1}{y}-\frac{1}{n}} f(y)\dd y,
\end{equation}
and it follows that $f$ is continuously differentiable on $(1,n)$, with
\begin{equation}\label{f'}
\gl f'(x)=(Tf)'(x)=
-\frac{1}{x^2}\int_0^{ x} f(y)\dd y,
\qquad 1<x<n.
\end{equation}

Conversely, if $f$ is continuous on $\on$, constant on $\oi$ and satisfies
\eqref{f'} on $(1,n)$ with the boundary condition $f(n)=0$, then $Tf=\gl f$.

Letting $F(x):=\int_0^x f(y)\dd y$, we have $F'(x)=f(x)$, and thus
\eqref{f'} yields the differential equation
\begin{equation}
  \label{F''}
  F''(x)=-\gl\qw x\qww F(x),
\qquad 1<x<n.
\end{equation}
Furthermore, 
$F(1)=\intoi f(x)\dd x = f(1)=F'(1)$ and $  F'(n)=f(n)=0$.
Hence, we have the boundary conditions (with derivatives at the
endpoints 1 and $n$ interpreted by continuity)
\begin{align}
  F'(1)&=F(1), \label{Fbc1}\\
  F'(n)&=0.\label{Fbc2}
\end{align}
Conversely, if $F$ solves \eqref{F''} on $(1,n)$ with the boundary
conditions \eqref{Fbc1}--\eqref{Fbc2}, then $f(x):=F'(x\lor 1)$ 
solves \eqref{f'} and
$\gl$ is an eigenvalue of $T$.

For a given $\gl>0$, the differential equation \eqref{F''} has the solutions
\begin{equation}\label{FF}
  F(x) = A x^{\ga_+} + B x^{\ga_-},
\end{equation}
where $\ga_\pm$ are the solutions of $\ga(\ga-1)=-\gl\qw$, and thus
\begin{equation}\label{a+-}
  \ga_\pm = \frac12\pm \sqrt{\frac{1}4-\frac{1}{\gl}}.
\end{equation}
If $\gl=4$, so we have a double root $\ga_+=\ga_-=\frac12$, we instead have
the solutions
\begin{equation}\label{Fdouble}
  F(x)=Ax\qq + B x\qq\ln x.
\end{equation}

Suppose that $\gl>0$ with $\gl\neq4$.
It is easily verified that the solutions \eqref{FF} that satisfy
\eqref{Fbc1} are multiples of $F(x):=\ga_+x^{\ga_+}-\ga_-x^{\ga_-}$.
Hence, $\gl$ is an eigenvalue of $T$ if and only if this function
satisfies \eqref{Fbc2}, i.e., if and only if
\begin{equation}\label{kod}
  \ga_+^2 n^{\ga_+-1} =   \ga_-^2 n^{\ga_--1}.
\end{equation}
Furthermore, then this eigenvalue is simple.

Consider first the case $0<\gl<4$. Then \eqref{a+-} yields the complex roots
$\ga_\pm=\frac12\pm \go\ii$, with $\go=\sqrt{1/\gl-1/4}$ and thus
\begin{equation}\label{kakk}
  \gl=\frac{1}{\go^2+\frac{1}{4}}=\frac{4}{1+4\go^2}.
\end{equation}
We rewrite \eqref{kod} as
\begin{equation}
\Bigparfrac{\frac12+\go\ii}{\frac12-\go\ii }^2e^{2\go\ln n\,\ii }=1,
\end{equation}
or, taking  logarithms,
\begin{equation}\label{ros}
  4\Im\ln \bigpar{1+2\go\ii} + 2\go \ln n \in 2\pi \bbZ.
\end{equation}
The \lhs{} of \eqref{ros} is a continuous increasing function of
$\go\in\ooo$, with the value 0 for $\go=0$.
Hence, for a given $n\ge2$,
there is for each integer $k\ge1$ exactly one solution $\go_k>0$ with
\begin{equation}\label{rosk}
  4\Im\ln \bigpar{1+2\go_k\ii} + 2\go_k\ln n = 2\pi k,
\end{equation}
and it follows, by \eqref{kakk}, that the eigenvalues of $T$ in $(0,4)$ are
\begin{equation}\label{glk}
  \gl_k:=\frac{4}{4\go_k^2+1},
\qquad k=1,2,\dots.
\end{equation}
In fact, these are all the non-zero eigenvalues,
since if $\gl>4$, so $\ga_\pm$ are real with $\ga_+>\ga_-$, then \eqref{kod}
cannot hold, and a similar argument shows that no non-zero $F$ of the form
\eqref{Fdouble} satisfies \eqref{Fbc1}--\eqref{Fbc2}.
(This also follows from \refR{RToo} below.)
Hence, \eqref{mgfhL} shows that, for every $t>-1/8$, at least,
\begin{equation}\label{mgfq}
  \E e^{-t\hL_n}=\prodk \bigpar{1+2\gl_k t}\qqw
=\prodk \Bigpar{1+\frac{8t}{1+4\go_k^2}}\qqw.
\end{equation}

Note that $\Im\ln(1+2\go_k\ii)\in(0,\pi/2)$, and thus \eqref{rosk} yields
\begin{equation}
  \label{qk}
\frac{\pi}{\ln n}(k-1) <\go_k < \frac{\pi}{\ln n}k.
\end{equation}

\begin{remark}\label{RToo}
The norm of $T=T_n$ is $\gl_1=4/(1+4\go_1^2)=4-O(1/\ln^2 n)$, see \eqref{glk}
and \eqref{qk}.
If we replace the lower cutoff $1$ in \eqref{T} by $a$, which by homogeneity
and a change of 
variables is equivalent to considering $T_{n/a}$, and then let $a\to0$ and
$n\to\infty$, we obtain as a weak limit of $T$ the integral operator
on $L^2$ with kernel $1/(x\lor y)$.
This limiting operator $\Too$
is bounded on $L^2\ooo$ with norm $4$, but it is not compact and has no
eigenvectors.  That the norm is  $4$ follows from
the result for $T_n$ above; that it is at most $4$ follows also from
\cite[Theorem 319]{HLP}; that there are no eigenvectors in $L^2\ooo$ is seen
by a direct calculation similar to the one above;
that $\Too$ is bounded by not compact follows also from 
\cite[Theorems 3.1 and 3.2]{SJ139}, where a class of integral
operators (including both $\Too$ and  $T_n$) is studied.
\end{remark}

\subsection{Asymptotics of the moment generating function}

So far we have kept $n$ fixed.
Now consider asymptotics as \ntoo.
Taking logarithms in \eqref{mgfq}, and using \eqref{qk}, we obtain for $t>0$
\begin{multline}\label{qb}
\frac12  \sumk \ln \Bigpar{1+\frac{8t}{1+(4\pi^2/\ln^2 n)k^2}} 
< 
-\ln \E e^{-t\hL_n}
\\
<\frac12   \sumko \ln \Bigpar{1+\frac{8t}{1+(4\pi^2/\ln^2 n)k^2}}.
\end{multline}
For $-1/8<t<0$, \eqref{qb} holds with the inequalities reversed.
Hence, for a fixed $t>-1/8$, uniformly in $n$,
\begin{equation}\label{qc}
  \begin{split}
\ln \E e^{-t\hL_n}
&= -\frac12  \sumk \ln \Bigpar{1+\frac{8t}{1+(4\pi^2/\ln^2 n)k^2}} + O(1)
\\&
=   -\frac12\intoo \ln \Bigpar{1+\frac{8t}{1+(4\pi^2/\ln^2 n)x^2}}\dd x + O(1).
\\&
=  -\frac{\ln n}{4\pi} \intoo \ln \Bigpar{1+\frac{8t}{1+ y^2}}\dd y + O(1).
  \end{split} 
\end{equation}
Furthermore,
\begin{align}
& \intoo \ln \Bigpar{1+\frac{8t}{1+ y^2}}\dd y 
=
 \intoo \Bigpar{\ln \bigpar{1+8t+ y^2}- \ln \bigpar{1+ y^2}}\dd y 
\notag\\&\qquad
\begin{aligned}
=\Bigl[
y\bigpar{\ln(1+8t+y^2)-\ln(1+y^2)}+2\sqrt{1+8t}\arctan(y/\sqrt{1+8t})\quad&
\\
-2\arctan(y)
\Bigr]_0^\infty &
\end{aligned}
\notag\\&\qquad 
=\pi\bigpar{\sqrt{1+8t}-1}.  \label{qd}
\end{align}
Consequently, we have shown, by \eqref{qc} and \eqref{qd}:

\begin{theorem}\label{TL2}
  For any fixed $t>-1/8$, and all $n\ge2$,
  \begin{equation}
\ln\E e^{-t\hL_n}
    =\frac{1-\sqrt{1+8t}}4\ln n + O(1).
  \end{equation}
\end{theorem}

For $t=-1/8$, a little extra work shows that \eqref{qc} holds with the error
term $O(\ln \ln n)$. If $t<-1/8$, then $-2t\gl_1>1$ for large $n$, and thus
$\E e^{-t\hL_n}=\infty$.

\subsection{A second proof of Theorem \ref{t1}}
\label{sec:new_proof}

By \refT{TL2} (and the comments after it, for completeness),
\begin{equation}\label{LDP0}
  \lim_\ntoo\frac{1}{\ln n} \ln \E e^{t \hL_n}
=
\gL(t):=
\begin{cases}
\frac{1-\sqrt{1-8t}}4, & t\le \xfrac{1}{8},\\
+\infty, & t >1/8.
\end{cases}
\end{equation}

The Legendre transform of $\gL(t)$ is, by a simple calculation,
\begin{equation}\label{gLx}
  \gLx(x):=\sup_{t\in\bbR} \bigpar{tx-\gL(t)}
=
\begin{cases}
  \frac{1}{8x}(x-1)^2 = \frac{x}{8}+\frac{1}{8x}-\frac{1}4, & x>0,\\
+\infty, & x\le 0.
\end{cases}
\end{equation}

By \eqref{LDP0} and
the G\"artner--Ellis theorem, see \eg{} \cite[Theorem 2.3.6]{DZ}
(and Remark (a) after it), the large deviation principle holds for the
variables $\hL_n/\ln n$ with rate function $\gLx(x)$ in \eqref{gLx}, 
in the sense that, for example,
\begin{equation}\label{LDP1}
\lim_\ntoo \frac{\ln \P(\hL_n\le c\ln n)}{\ln n} = -\gLx(c), \qquad 0<c\le1.
\end{equation}
Note that $\gLx(c)=K(c)$ given by \eqref{3a}.
Consequently,
we have shown the following Browian analogue of Theorem \ref{t1}.

\begin{theorem} \label{TB1}
For every $c\in(0,1]$,
  \begin{equation}\label{LDP2}
  \P(\hL_n\le c\ln n) =n^{-K(c)+o(1)}.
\end{equation}
\end{theorem}

\begin{proof}[Second proof of \refT{t1}]
By \refT{TB1} and \refL{L5}.
\end{proof}

Moreover, \eqref{LDP0} and the G\"artner--Ellis theorem
give also a corresponding result for the upper tail.
\begin{theorem} \label{TB2}
For every $c\in [1,\infty)$,
  \begin{equation}\label{LDP3}
  \P(\hL_n\ge c\ln n) =n^{-K(c)+o(1)}.
\end{equation}
\end{theorem}

This result too transfers from the Brownian version to the random walk.
\begin{theorem} \label{TU}
For every $c\in [1,\infty)$,
  \begin{equation}\label{tu}
  \P(L_n\ge c\ln n) =n^{-K(c)+o(1)}.
\end{equation}
\end{theorem}

\begin{proof}
  This follows by \refT{TB2} and an upper tail version of \refL{L5} with
$\Pr(L_n \le c\ln  n)$ replaced  by $\Pr(L_n\ge c\ln n)$, and so on; this
version is proved in the same way as above, so we omit the details.
\end{proof}

\section{Conditional functional limit laws}\label{section:conditionallaw}

In this section, we  study the preferential attachment process
$\{(X_k,Y_k): k \geq 2\}$
defined in \refS{sec:model}, and
establish functional limit theorems for the trajectories, conditional on the
event $BINGO(n,n)$. 
We define $\gD_k:=X_k-Y_k$, so that the process is given by \eqref{XYgD},
and state the results in terms of the 
stochastic process $\{ \Delta_k : 2 \leq k \leq 2n \}$, conditional on 
$BINGO(n,n)$; recall that $BINGO(n,n)$ 
in this notation is the event $\gD_{2n}=0$.

In particular, we prove \refT{Toi} stated in \refS{sec:model}. We also state
and prove related functional limit results for the process at times $o(n)$.
We establish the results using the usual two step approach--- 
first, we establish finite dimensional convergence, and then establish
tightness (see e.g.\ \cite{Billingsley}).
The proofs proceed using the local CLT  estimates in \refS{sec:cont_time2},
in particular \refT{Tbingo}. Finite-dimensional convergence follows by
straightforward calculations, but
our proof of tightness is rather complicated, and uses several lemmas.
We base the proof of tightness on a theorem by Aldous
\cite{Aldous-tightness}, see 
\refS{SS:tightness} below, but for technical resons discussed there, we do
not use Aldous's result directly. Instead, we state and prove in
\refS{SS:tightness} a variant of it that is convenient in our situation.
We then prove \refT{Toi} in \refS{SS:pfToi}, and give corresponding results
for small times in \refS{SS:on}.

Note that the processes $(X_k,Y_k)$, $\gD_k$ and $n\qqw\gD_{\floor{2nt}}$
are  Markov processes, 
and so they are (by a simple, general, calculation) 
also conditioned on $BINGO(n,n)$.

\subsection{A general criterion for tightness}\label{SS:tightness}

Our proof uses a tightness criterion by Aldous \cite{Aldous-tightness}
(and, in a slightly different formulation, Mackevi{\v c}ius
\cite{Mackevicius}), see also 
\cite[Lemma 3.12]{whitt_fclt}. 
 Recall  that
a sequence of $\Doo$-valued stochastic processes 
$\{ Z_n(t) : n \geq 1 \}$ is 
\emph{stochastically bounded} if for every $T >0$, 
\begin{align}\label{SB}
\lim_{M \to \infty} \sup_{n } \P\Bigsqpar{ \max_{0 \le t \le T } | X_n(t)| >  M}  =0. 
\end{align}
It is well-known, and easy to see, that it suffices to show \eqref{SB} with 
$\sup_n$ replaced by $\limsup_\ntoo$.

\begin{lemma}[\cite{Aldous-tightness,Mackevicius,whitt_fclt}]\label{LA}
Suppose that $Z_n(t)$ is a sequence of stochastic processes in $\Doo$ 
satisfying the following conditions.
\begin{romenumerate}
\item \label{LAa}
$\set{Z_n(t):n\ge1}$ is stochastically bounded.
\item \label{LAb}
For each $n\ge1$, $T >0$, $\varepsilon >0$, $\lambda < \infty$ and $\gd>0$,
there exists a number
$
\alpha_n(\lambda, \varepsilon, \delta, T)$
such that
\begin{align}\label{la}
\Pr\bigsqpar{ | Z_n(u) - Z_n(t_m)| > \varepsilon \bigm| Z_n(t_1), \cdots,
  Z_n(t_m) } 
\leq \alpha_n(\lambda, \varepsilon, \delta, T)
\end{align}
\as{} on the event 
$\{\max_i |Z_n(t_i)| \leq \lambda \}$, 
for every finite sequence $\{t_i : 1 \leq i \leq m\}$ 
and $u$
with $0 \leq t_1 \leq t_2 \leq \cdots \leq t_m \leq u \leq T$ and 
$u - t_m \le \delta$.  
Furthermore, these numbers $\ga_n(\gl,\eps,\gd,T)$ satisfy
\begin{align}\label{laga}
\lim_{\delta \downarrow 0} 
 \limsup_{n \to \infty} \alpha_n (\lambda, \varepsilon, \delta, T) =0, 
\end{align}
for every $\gl,T,\eps$.
\end{romenumerate}
Then the sequence $Z_n(t)$ is tight in $\Doo$.
\end{lemma}

For Markov processes (as in our case),
the condition \eqref{la} simplifies: 
by the Markov property,
it suffices to consider the case $m=1$.

A technical problem that prevents us from a direct application of \refL{LA}
to our processes, using \refT{Tbingo} to verify the condition,
is that in \eqref{la}, $u-t_m$ may be
arbitrarily small, while in \refT{Tbingo}, $B/A$ is supposed to be
bounded below by some $\cg>1$. We thus first prove the following variant of
\refL{LA}, where we have a lower bound on $u-t_m$. 
For simplicity, we state the lemma only in the Markov case.
We assume also, again for simplicity, 
that the processes are strong Markov; recall that this means,
informally, that the Markov property holds not only at fixed times, but also
at stopping times. A discrete-time Markov process, or a process such as our
$n\qqw\gD_{\floor{2nt}}$ that essentially has discrete time, 
is automatically strong Markov.

The main difference from \refL{LA} is that the condition $0\le u-t_m\le\gd$
is replaced by $\gd\le u-t\le 2\gd$.
We also add a condition that the jumps are uniformly bounded 
(which trivially holds in our case); 
we do not know whether this condition really is needed.
(The condition can presumably be weakened to stochastic boundedness of the
jumps, as in \cite{Billingsley-tightness}, but we have not pursued this.)

\begin{lemma}\label{LT}
Suppose that $Z_n(t)$ is a sequence of strong 
Markov processes in $\Doo$ 
satisfying the following conditions.
\begin{romenumerate}
\item \label{LTa}
$\set{Z_n(t):n\ge1}$ is stochastically bounded.
\item \label{LTb}
For each $n\ge1$, $T >0$, $\varepsilon >0$, $\lambda < \infty$ and $\gd>0$,
there exists a number
$
\alpha_n(\lambda, \varepsilon, \delta, T)$
such that
\begin{align}\label{lt1}
\Pr\bigsqpar{| Z_n(u) - Z_n(t)| > \varepsilon \bigm| 
Z_n(t) } 
\leq \alpha_n(\lambda, \varepsilon, \delta, T)
\end{align}
\as{} on the event 
$\set{ |Z_n(t)| \leq \lambda }$,
for every $t$ and $u$ with $0\le t\le u\le T$ and $t+\gd\le u\le t+2\gd$.
Furthermore, these numbers $\ga_n$ satisfy
\begin{align}\label{lt2}
\lim_{\delta \downarrow 0} 
 \limsup_{n \to \infty} \alpha_n (\lambda, \varepsilon, \delta, T) =0, 
\end{align}
for every $\gl,T,\eps$.
\item \label{LTc}
The jumps are bounded by $1$:
\begin{equation}\label{lt3}
|  Z_n(t)-Z_n(t-)|\le1
\end{equation}
for all $n$ and $t$.
\end{romenumerate}
Then the sequence $Z_n(t)$ is tight in $\Doo$.
\end{lemma}

We reduce to \refL{LA} using the following lemma.

 \begin{lemma}\label{LM}
   Suppose that $Z(t)$ is a strong Markov process in $\Doo$,
such that for some given numbers $\gl,T,\eps,\gd,\ga>0$, 
\begin{equation}
  \label{lm1}
\Pr\bigsqpar{|Z(u)-Z(t)|\ge\eps\mid Z(t)} \le\ga
\end{equation}
\as{} on the event $\set{|Z(t)|\le\gl+2\eps+1}$,
for
each $t$ and
$u$ with $0\le t\le u\le T+2\gd$ with $t+\gd\le u\le t+2\gd$.
Suppose further that the jumps in $Z(t)$ are bounded by $1$, i.e.,
\begin{equation}
  \label{lm2}
|Z(t)-Z({t-})|\le 1
\end{equation}
for all $t\ge0$.
Then, for each $t\le T$,
\begin{equation}\label{lm00}
  \Pr\Bigsqpar{\sup_{u\in[t,t+\gd]}|Z(u)-Z(t)|>2\eps\mid Z(t)}
\le 2\ga
\end{equation}
\as{} on the event $\set{|Z(t)|\le\gl}$.
 \end{lemma}

\begin{proof}
Let $\cF_t$ be the $\gs$-field generated by \set{Z(s):s\le t}, and
define a stopping time by
\begin{equation}
  \label{tau}
\tau:=\inf\bigset{u\in[t,t+\gd]:|Z(u)-Z(t)|\ge2\eps},
\end{equation}
using the definition $\inf\emptyset:=\infty$ if there is no such $u$.

Let $v:=t+2\gd$. If $\tau<\infty$, then 
$\tau\in[t,t+\gd]$, and thus
$\tau+\gd\le v\le \tau+2\gd$;
hence,
by \eqref{lm1} and the strong Markov property
\begin{equation}
  \label{lm3}
\Pr\bigsqpar{|Z(v)-Z(\tau)|>\eps\mid\cF_{\tau}}\le\ga
\end{equation}
\as{} on the event $\set{\tau<\infty}\cap\set{|Z(\tau)|\le \gl+2\eps+1}$.
Furthermore, $\tau<\infty$ implies, by the definition \eqref{tau} and
right-continuity,
\begin{equation}
  \label{lm4}
|Z(\tau)-Z(t)|\ge2\eps, 
\end{equation}
and thus also $\tau>t$ and, by \eqref{lm1} again,
\begin{equation}
  \label{lm5}
|Z({\tau-})-Z(t)|\le2\eps.
\end{equation}

Let $\cE$ be any event with $\cE\in\cF_t$ and 
$\cE\subseteq \set{|Z(t)|\le \gl}$.
Then, on the event  $\cE\cap\set{\tau<\infty}$,
by \eqref{lm2} and \eqref{lm5},
\begin{equation}
  \label{lm6}
|Z(\tau)|
\le
|Z(t)|+
|Z(\tau-)-Z(t)|+
|Z(\tau)-Z(\tau-)|
\le \gl+ 2\eps+1.
\end{equation}
Hence, \eqref{lm3} applies, and thus, since $\cF_t\subseteq\cF_\tau$,
\begin{equation}
\Pr\bigsqpar{\set{|Z(v)-Z(\tau)|\le\eps}\cap\cE\cap\set{\tau<\infty}}
\ge(1-\ga)\Pr\bigsqpar{\cE\cap\set{\tau<\infty}}.
\end{equation}
In other words, recalling that $\cE$ can be any event in $\cF_t$
with
$\cE\subseteq \set{|Z(t)|\le \gl}$,
\begin{equation}
  \label{lm8}
\Pr\bigsqpar{\set{|Z(v)-Z(\tau)|\le\eps}\cap\set{\tau<\infty}\mid\cF_t}
\ge(1-\ga)\Pr\bigsqpar{\set{\tau<\infty}\mid\cF_t}
\end{equation}
\as{} on the event $\set{|Z(t)|\le \gl}$.

Furthermore, $\tau<\infty$ and $|Z(v)-Z(\tau)|\le\eps$  imply, using
  \eqref{lm4},
$|Z(v)-Z(t)|\ge\eps$. Consequently, \eqref{lm1} implies
(using the Markov property)
\begin{equation}
  \label{lm9}
\Pr\bigsqpar{\set{|Z(v)-Z(\tau)|\le\eps}\cap\set{\tau<\infty}\mid\cF_t}
\le
\Pr\bigsqpar{\set{|Z(v)-Z(t)|\ge\eps}\mid\cF_t}\le\ga
\end{equation}
\as{} on the event $\set{|Z(t)|\le \gl}$.

Assume first $\ga\le1/2$.
Combining \eqref{lm8} and \eqref{lm9}, we obtain
\begin{equation}
  \label{lm10}
\Pr\bigsqpar{\set{\tau<\infty}\mid\cF_t}
\le \frac{\ga}{1-\ga}
\le2\ga,
\end{equation}
\as{} on the event $\set{|Z(t)|\le \gl}$.
Since the event
$\set{\sup_{u\in[t,t+\gd]}|Z(u)-Z(t)|>2\eps}\subseteq\set{\tau<\infty}$,
the result follows.
The case $\ga>1/2$ is trivial.
\end{proof}

\begin{proof}[Proof of \refL{LT}]
\refL{LM} applies to each $Z_n$ with 
$\eps$ replaced by $\eps/2$ 
and
$\ga:=\ga_n(\gl+\eps+1,\eps/2,\gd,T+2\gd)$.
This shows that,
for each $t\le T$,
\begin{multline}\label{lm0}
  \Pr\Bigsqpar{\sup_{u\in[t,t+\gd]}|Z_n(u)-Z_n(t)|>\eps\mid Z_n(t)}
\\
\le \ga'_n(\gl,\eps,\gd,T):=2\ga_n(\gl+\eps+1,\eps/2,\gd,T+2\gd).
\end{multline}
\as{} on the event \set{|Z_n(t)|\le\gl}.
Hence, the assumption \ref{LAb} holds with $\ga_n$ replaced by $\ga_n'$; 
note that \eqref{laga} holds for $\ga_n'$ by \eqref{lt2} and \eqref{lm0}
(since it suffices to consider $\gd\le1$).

Hence,  \refL{LA} applies and the result follows.
\end{proof}

\subsection{Proof of \refT{Toi}}\label{SS:pfToi}

Note first that if we define $\Ga(t)$ by \eqref{G-Br}, then its covariance
function agrees with \eqref{toi1}; this shows that the Gaussian process
$\Ga(t)$ in \refT{Toi} really exists and is continuous on \oi.
(Also at $t=0$, since $\Br(t)$ is H\"older$(\frac12-\eps)$ for every $\eps>0$.)
Equivalently, we can define $\Ga(t)$ from a Brownian motion $B(t)$ by either
\begin{equation}\label{G-B1}
  \Ga(t):=(\ga-1/2)\qw\bigpar{t^{1-\ga}-t^{\ga}}B_{t^{2\ga-1}/(1-t^{2\ga-1})},
\qquad 0< t<1,
\end{equation}
or (reversing the flow of time)
\begin{equation}\label{G-B2}
  \Ga(t):=(\ga-1/2)\qw t^{\ga}B_{t^{1-2\ga}-1},
\qquad 0< t\le1.
\end{equation}
Again, these are verified by calculating the covariances.
Note that $\Ga(0)=\Ga(1)=0$, \eg{} by \eqref{G-Br}.

\begin{lemma}\label{Lfd}
  Finite-dimensional convergence holds in \eqref{toi2}, i.e., if $0\le
  t_1<\dots<t_m\le 1$ are fixed, then,
conditioned on $BINGO(n,n)$, as \ntoo,
  \begin{equation}\label{lfd}
        n\qqw\bigpar{ \gD_{\floor{2nt_1}},\dots,\gD_{\floor{2nt_m}}}
\dto \bigpar{\Ga(t_1),\dots,\Ga(t_m)}.
  \end{equation}
\end{lemma}

\begin{proof}
Since $\gD_0=\gD_{2n}=0$ and $\Ga(0)=\Ga(1)=0$ by definition, 
we may assume $0<t_1<\dotsm<t_m<1$.
We fix also some $M>0$.
Let $\nnn{i}:=\floor{2nt_i}$,
and let $k_1,\dots,k_m\in\bbZ$ with $\nnn{i}+k_i\in2\bbZ$
and $|k_i|\le M \sqrt n$.
(Assume $n$ so large that each $n_i\ge2$.)
Then $BINGO(n,n)$ holds together with $\gD(n_i)=k_i$ for $i=1,\dots,m$, 
if and only if the events $\cA_1,\dots \cA_{m+1}$ occur,
where we set 
\begin{align}\label{b1}
\mathcal{A}_1 
 &= BINGO\Bigpar{\frac{\nnn1 + k_1 }{2}, \frac{ \nnn1 - k_1 }{2}},
\intertext{and for $2\le i\le m+1$, with $n_{m+1}:=2n$ and $k_{m+1}:=0$,}
\mathcal{A}_i &= 
BINGO\Bigpar{\frac{\nnn{i-1} + k_{i-1} }{2}, \frac{\nnn{i-1} - k_{i-1} }{2};
\frac{\nnn{i} + k_{i} }{2}, \frac{\nnn{i} - k_{i} }{2}}.
\label{b2}
\end{align}
The events $\cA_1,\dots,\cA_{m+1}$ are independent, and thus
\begin{align}\label{b3}
&\P\bigsqpar{\Delta( \nnn{1}) = k_1,\dots, \Delta(\nnn{m})= k_m \mid BINGO(n,n)}
\nonumber
\\&\qquad
 = \frac{\prod_{i=1}^{m+1} \P[ \mathcal{A}_i] }{\P [ BINGO(n,n) ]}
 = \frac{\P[ \mathcal{A}_1] }{\P [ BINGO(n,n) ]} 
\prod_{i=2}^{m+1} \P[ \mathcal{A}_i] .
\end{align}
By \refT{t3}, and $|k_1|\le M\sqrt n \le M'\sqrt{n_1}$,
\begin{align}
\frac{\P[ \mathcal{A}_1] }{\P [ BINGO(n,n) ]} 
&=
\frac{\P\Bigsqpar{BINGO\bigpar{\frac{\nnn1 + k_1 }{2}, \frac{ \nnn1-k_1 }{2}}} }
{\P [ BINGO(n,n) ]} 
\nonumber\\&
\sim
\frac{\bigpar{\frac{\nnn1+k_1}{2}}^{-\ga}+\bigpar{\frac{\nnn1-k_1}{2}}^{-\ga}}
{2n^{-\ga}}
\sim
\frac{2(\nnn1/2)^{-\ga}}{2n^{-\ga}}
\sim t_1^{-\ga},
\label{b4}
\end{align}
where, as in the estimates below, the implicit factors $1+o(1)$ tend to 1
as \ntoo,
uniformly for all 
$k_1,\dots,k_m$ as above, for fixed $t_1,\dots,t_m$ and $M$.

Denote the probability density function of the normal distribution $N(0,t)$ by
\begin{equation}\label{phit}
  \phi_t(x) := (2\pi t)\qqw e^{-x^2/2t}.
\end{equation}
Furthermore, let 
\begin{align}
\kk&:=\sqrt{\ga-1/2},   
\label{kk}\\
  T_i&:=t_i^{1-2\ga}-1,  \label{Ti}
\\
y_i&:= \kk t_i^{-\ga}k_i/\sqrt n.  \label{yi}
\end{align}
Note that $n_i/2\sim t_in$, where we define
$t_{m+1}:=1$.
Thus \eqref{b2} and \refC{Cbingo2} yield,
for $2\le i\le m+1$,
\begin{align}
 \Pr[ \cA_i]
&\sim\sqrt{\frac{2\ga-1}{\pi}}
\frac{1}{n\qq t_i^{\ga}\sqrt{t_{i-1}^{1-2\ga}-t_i^{1-2\ga}}}
\exp\biggpar{-\frac{\bigpar{y_{i-1}-y_i}^2}{2(t_{i-1}^{1-2\ga}-t_i^{1-2\ga})} }
\nonumber\\
&=\frac{2\kk}{t_i^{\ga}\sqrt n}
\phi_{T_{i-1}-T_i}(y_{i-1}-y_i)
.\label{b5}
\end{align}
Consequently, \eqref{b3}, \eqref{b4} and \eqref{b5} yield,
uniformly for $|k_i|\le M\sqrt n$,
\begin{align}\label{b6}
&\P\bigsqpar{\Delta( \nnn{1}) = k_1,\dots, \Delta(\nnn{m})= k_m \mid BINGO(n,n)}
\nonumber
\\&\qquad
\sim 
t_1^{-\ga} \prod_{j=1}^m \frac{2\kk}{t_{j+1}^{\ga}\sqrt n}
\phi_{T_j-T_{j+1}}(y_j-y_{j+1})
=
 \prod_{j=1}^m \frac{2\kk}{t_{j}^{\ga}\sqrt n}
\phi_{T_j-T_{j+1}}(y_j-y_{j+1})
.
\end{align}

Note that $T_1>\dots>T_m>T_{m+1}=0$, and thus (recalling $y_{m+1}=0$)
$\prod_{i=1}^m\phi_{T_j-T_{j+1}}(y_j-y_{j+1})$ is the joint density function 
of $\bigpar{B(T_1),\dots,B(T_m)}$ for a Brownian motion $B(t)$.
Since $M$ is arbitrary, it follows easily, recalling the scaling \eqref{yi} and
noting that 
$k_i+n_i\in 2\bbZ$, so $k_i$ takes values spaced by 2, and thus $y_i$ takes
values spaced by $2\kk t_i^{-\ga}n^{-1/2}$, that
\begin{equation}
n\qqw  \bigpar{\kk t_1^{-\ga}\gDp{n_1},\dots,\kk t_m^{-\ga}\gDp{n_m}}
\dto \bigpar{B(T_1),\dots,B(T_m)}.
\end{equation}
Hence, 
\begin{equation}\label{krhmf}
  \begin{split}
    n\qqw  \bigpar{ \gDp{n_1},\dots, \gDp{n_m}}
&\dto \bigpar{\kk\qw t_1^{\ga}B(T_1),\dots,\kk\qw t_m^{\ga}B(T_m)}
  \end{split},
\end{equation}
which 
using \eqref{Ti} and the construction \eqref{G-B2} is the same as \eqref{lfd}.
\end{proof}

\begin{remark}\label{Runiform}
  We assumed in the proof above that $t_1,\dots,t_m$ are fixed.
In fact, the proof shows, using the uniformity assertion in \refT{Tbingo}, 
that the estimates, in particular \eqref{b6}, hold
uniformly for 
all $0<t_1<\dots<t_m<1$ with $\min\set{t_1,t_{i+1}-t_i, 1-t_m}\ge\delta$, for 
any fixed $M<\infty$ and  $\delta>0$.
\end{remark}

We let $Z_n$ denote the processes $Z_n(t):=n\qqw \gDnt$, 
always conditioned on $BINGO(n,n)$. 
We also let $\PrB$ denote probabilities conditional on $BINGO(n,n)$.

We proceed to tightness and functional convergence.
Our proof requires special arguments for $t$ close to the endpoints 0 and 1,
mainly because the lack of uniformity in \refT{Tbingo} when $A$ is close to
0 or $B$. We first prove that the sequence $Z_n(t)$ is stochastically
bounded on a subinterval $[a,b]\subset(0,1)$.

\begin{lemma}\label{LSB}
  Let $0<a<b<1$. Then, conditioned on $BINGO(n,n)$, 
the sequence of stochastic processes
  $Z_n(t):=n\qqw\gD_{\floor{2nt}}$ is stochastically bounded on $[a,b]$.
\end{lemma}

\begin{proof}
Let $A_0:=\floor{2na}$ and $B_0:=\floor{2nb}$.
Further, let $K>0$ be a large number.
Define (for each $n$) the stopping time   
\begin{equation}\label{tauk}
  \tau_K:=\inf\set{k\ge A_0:|\gD_k| \ge Kn\qq},
\end{equation}
as always with $\inf\emptyset:=\infty$.
Then
\begin{align}\label{kaj}
&  \Pr\bigsqpar{\set{A_0<\tau_K\le B_0 }\land BINGO(n,n)}
\nonumber\\&\qquad=
\sum_{k=A_0+1}^{B_0}\P[\tau_K=k] \Pr\bigsqpar{BINGO(n,n)\mid \tau_K=k}.
\end{align}
If $\tau_K=k>A_0$, then $|\gD_k|=\ceil{Kn\qq}$ or $\ceil{Kn\qq}+1$ (depending
on the parity of $k$). Denoting this number by $\gDx$,
and assuming $A_0<k\le B_0$,
we have by \refC{Cbingo2}, 
for all $n\ge n_0$ for some $n_0$ not depending on $k$,
and some $C$ not depending on $n$, $k$ or $K$ (but perhaps on $a,b,\ga$),
\begin{align}
&\Pr\bigsqpar{BINGO(n,n)\mid \tau_K=k}
= \Pr\Bigsqpar{BINGO\Bigpar{\frac{k\pm\gDx}2,\frac{k\mp\gDx}2 ;n,n}} 
\nonumber\\
&\qquad
\le C n\qqw e^{-\gDx^2/2k}
\le C n\qqw e^{-K^2 n/2k}
\le C n\qqw e^{-K^2/4}.
\label{kai}
\end{align}
Note also that $\tau_K=k>A_0$ implies $|\gD_{A_0}|<Kn\qq$ by \eqref{tauk}.
Hence, \eqref{kaj} and \eqref{kai} yield, for $n\ge n_0$,
\begin{align}\label{kak}
&  \Pr\bigsqpar{\set{A_0<\tau_K\le B_0 }\land BINGO(n,n)}
\nonumber\\&\qquad
\le
\sum_{k=A_0+1}^{B_0}\P[\tau_K=k] Cn\qqw e^{-K^2/4}
\nonumber\\&\qquad
= Cn\qqw e^{-K^2/4} \P\bigsqpar{A_0<\tau_K\le B_0}
\nonumber\\&\qquad
\le C  n\qqw e^{-K^2/4} \P\bigsqpar{|\gD_{A_0}|<Kn\qq}.
\end{align}
Furthermore, by \refT{t3}, as \ntoo,
\begin{align}
  \P\bigsqpar{|\gD_{A_0}|<Kn\qq}
&
=\sum_{|j|<Kn\qq} \Pr\Bigsqpar{BINGO\Bigpar{\frac{A_0+j}2,\frac{A_0-j}2 }} 
\nonumber\\&
\sim 2Kn\qq\cdot2\gb (A_0/2)^{-\ga}
\sim 4\gb K a^{-\ga} n^{1/2-\ga}.
\label{kam}
\end{align}
By the same theorem,  $\Pr\bigsqpar{BINGO(n,n)}\sim 2\gb n^{-\ga}$.
Consequently, \eqref{kak} implies
\begin{align}
&\limsup_\ntoo  
\PrB\bigsqpar{\set{A_0<\tau_K\le B_0 }}
\nonumber\\&\qquad
=\limsup_\ntoo
\frac{ \Pr\bigsqpar{\set{A_0<\tau_K\le B_0 }\land BINGO(n,n)}}
{\Pr\bigsqpar{BINGO(n,n)}}
\nonumber\\&\qquad
\le\limsup_\ntoo
\frac{C n\qqw e^{-K^2/4} \cdot
4\gb K a^{-\ga} n^{1/2-\ga}}
{2\gb n^{-\ga}}
= C_1 K e^{-K^2/4}.
\end{align}

We have also, by \refL{Lfd}, as \ntoo,
\begin{align}
  \PrB\bigsqpar{\tau_K=A_0 }
&= \PrB \bigsqpar{|\gD_{A_0}|\ge K n\qq}
\to \Pr\bigsqpar{G_\ga(a)\ge K}.
\end{align}
Consequently, conditioned on $BINGO(n,n)$,
\begin{align}
\limsup_\ntoo  \PrB\bigsqpar{\sup_{a\le t\le b}|Z_n(t)|\ge K}
&=
\limsup_\ntoo  \PrB\bigsqpar{A_0\le \tau_k\le B_0 }
\nonumber\\&
\le C_1 K e^{-K^2/4} +  \Pr\bigsqpar{G_\ga(a)\ge K}.
\end{align}
The right-hand side tends to 0 as $K\to\infty$, and the result follows.
\end{proof}

We next prove that $Z_n(t)$ is (with large probability)
uniformly small for small $t$.

\begin{lemma}\label{LSB0}
For every $\eps>0$ and $\eta>0$, there exists $\gd>0$ such that, 
\begin{align}\label{lsb0}
\sup_{n\ge1}
\PrBX{\sup_{0\le t\le \gd}|Z_n(t)| >\eps}\le\eta.
\end{align}
\end{lemma}

\begin{proof}
The argument is similar to the proof of \refL{LSB}, but with some differences.
Define (for each $n$) the stopping time   
\begin{equation}\label{taue}
  \taue:=\inf\set{k\ge 0:|\gD_k| \ge \eps n\qq}.
\end{equation}
Let $A_0:=\ceil{\eps n\qq}$ and
fix also some $M\ge \eps$.
Note that if $k\le A_0$, then $|\gD_k|\le k-2<\eps n\qq$.
Thus, $\taue>A_0$.

Suppose that 
\begin{align}
  \label{Ac}
A_0\le A\le \frac{\eps}{16M}n.
\end{align}
Then,
with $\gDx$ as in the proof of \refL{LSB} (with $\eps$ instead of $K$),
\begin{align}
&  \Pr\bigsqpar{\set{A<\taue\le 2A }\land\set{|\gD_n|\le Mn\qq}\land 
 BINGO(n,n)}
\nonumber\\&\qquad=
\sum_{k=A+1}^{2A}\sum_{|j|\le Mn\qq}\P[\taue=k]\times
\nonumber\\ &\hskip4em
\Pr\Bigsqpar{BINGO\Bigpar{\frac{k\pm\gDx}2,\frac{k\mp\gDx}2
    ;\frac{n+j}2,\frac{n-j}2}} \times
\nonumber\\ &\hskip6em
\Pr\Bigsqpar{BINGO\Bigpar{\frac{n+j}2,\frac{n-j}2;n,n}}.
\label{ll1}
\end{align}
Let $c$ and $C$ denote positive constants, 
not depending on $n$, $A$, $\eps$ or $M$, but
possibly varying from one occurence to another.
\refL{LQ} applies by \eqref{Ac}, 
provided $n\ge n_1(M)$ for some $n_1(M)$ depending on $M$ only,
and yields, for every $k$ in the sum in
\eqref{ll1},
\begin{align}
\sum_{|j|\le Mn\qq}
&  \Pr\Bigsqpar{BINGO\Bigpar{\frac{k+\gDx}2,\frac{k-\gDx}2
    ;\frac{n+j}2,\frac{n-j}2}} 
\le e^{-c \eps^2 n /2k}
\nonumber\\&\hskip2em
\le e^{-c \eps^2 n /A}.  
\end{align}
By symmetry the same holds with $\gDx$ replaced by $-\gDx$.
Similarly, \refC{Cbingo2} yields,
provided $n\ge n_2(M)$, 
\begin{align}
\Pr\Bigsqpar{BINGO\Bigpar{\frac{n+j}2,\frac{n-j}2;n,n}}
\le C n\qqw.
\end{align}
Hence, \eqref{ll1} implies, assuming from now on that 
$n\ge n_1(M)\lor n_2(M)$, 
\begin{align}
&  \Pr\bigsqpar{\set{A<\taue\le 2A }\land\set{|\gD_n|\le Mn\qq}\land 
 BINGO(n,n)}
\nonumber\\&\qquad
\le C\Pr\bigsqpar{A<\taue\le 2A} e^{-c\eps^2 n/A} n\qqw. 
\label{ll2}
\end{align}

If $A\ge 2A_0$, we
use \refT{t3} similarly to \eqref{kam} and obtain
\begin{align}\label{ll3}
  \Pr\bigsqpar{A<\taue\le 2A}
\le \Pr\bigsqpar{|\gD_A|< \eps n\qq}
\le C \eps n\qq A^{-\ga}.
\end{align}
Combining \eqref{ll2} and \eqref{ll3},
and recalling \eqref{bingo}, 
we then obtain
\begin{align}
  \PrB\bigsqpar{\set{A<\taue\le 2A }\land\set{|\gD_n|\le Mn\qq}}
&\le C \eps e^{-c\eps^2 n/A} A^{-\ga} n^{\ga} 
\nonumber\\&
\le C \eps^{1-2(\ga+1)} A/n.
\label{ll4}
\end{align}
If $A_0\le A<2A_0$, we instead use $  \Pr\bigsqpar{A<\taue\le 2A}\le1$ and
obtain from \eqref{ll2} similarly
\begin{align}
&  \PrB\bigsqpar{\set{A<\taue\le 2A }\land\set{|\gD_n|\le Mn\qq}}
\le C  e^{-c\eps^2 n/A} n\qqw n^{\ga} 
\nonumber\\&\qquad
\le C  e^{-c\eps n\qq}  n^{\ga-1/2}
\le C \eps^{-2\ga-1} A_0/n,
\label{ll4a}
\end{align}
yielding the same conclusion as \eqref{ll4}.

Assume $0<\gd\le \eps/32 M$ and sum \eqref{ll4} or \eqref{ll4a}
with $A=A_j:=2^jA_0$,
for $0\le j\le j_0:= \floor{\log_2(2\gd n/A_0)}$;
note that $\gd n < A_{j_0}\le 2\gd n$, and that \eqref{Ac} holds for each $A_j$.
This yields, 
with $C(\eps)$ denoting constants that may depend on $\eps$,
\begin{align}
  \PrB\bigsqpar{\set{A_0<\taue\le 2\gd n }\land\set{|\gD_n|\le Mn\qq}}
&\le C(\eps) 
\sum_{j=0}^{j_0} \frac{A_{j}}n
\le C(\eps) 
\gd
\label{ll5}
\end{align}
and thus, recalling that $\taue>A_0$,
\begin{align}
  \PrB\bigsqpar{\max\set{|\gD_k|:k\le 2\gd n }\ge\eps n\qq}
&=   \PrB\bigsqpar{{A_0<\taue\le 2\gd n }}
\nonumber\\&
\le C(\eps)\gd  
+   \PrB\bigsqpar{|\gD_n|> Mn\qq}. 
\label{ll6}
\end{align}
The right-hand side can be made smaller than $\eta$, uniformly in $n$, by
choosing $M$ large and $\gd$ small.
We assumed above that 
$n\ge n_1(M)\lor n_2(M)$. 
The result extends trivially to all $n$ as
stated in \eqref{lsb0} by decreasing $\gd$.
\end{proof}

\begin{lemma}\label{LSB1}
For every $\eps>0$ and $\eta>0$, there exists $\gd>0$ such that, 
\begin{align}\label{lsb1}
\limsup_\ntoo
\PrBX{\sup_{1-\gd\le t\le 1}|Z_n(t)| >\eps}\le\eta.
\end{align}
\end{lemma}

Unlike \refL{LSB0}, we can not here replace $\limsup_n$ by $\sup_n$,
since trivially $\sup_{t\ge 1-\gd}|Z_n(t)|\ge n\qqw$ for every $n$ and $\gd$.

\begin{proof}
Assume $0<\gd<\frac12$,
let $A_0:=\floor{(1-\gd)2 n}$ and
define (for each $n$) the stopping time   
\begin{equation}\label{taue1}
  \taue:=\inf\set{k\ge A_0:|\gD_k| \ge \eps n\qq}.
\end{equation}
Let $\gDx$, $C$ and $C(\eps)$ be as in the proof of \refL{LSB0}.
Then,
using \refL{LD},
\begin{align}
&  \Pr\bigsqpar{\set{A_0<\taue< 2n }\land BINGO(n,n)}
\nonumber\\&\qquad
=\sum_{k=A_0+1}^{2n-1}\P[\taue=k]
\Pr\Bigsqpar{BINGO\Bigpar{\frac{k\pm\gDx}2,\frac{k\mp\gDx}2
    ;n,n}}
\nonumber\\&\qquad
\le\sum_{k=A_0+1}^{2n-1}\P[\taue=k]
\frac{C}{\gDx}e^{-\eps^2 n/4(2n-k)}
\nonumber\\&\qquad
\le C(\eps)n\qqw e^{-\eps^2/8\gd}\P[\taue>A_0].
\label{ll11}
\end{align}
Furthermore, 
similarly to \eqref{kam},
\refT{t3} yields
\begin{align}\label{ll13}
  \Pr\bigsqpar{\taue >A_0}
= \Pr\bigsqpar{|\gD_{A_0}|< \eps n\qq}
\le C \eps n\qq A_0^{-\ga}
\le C \eps n^{1/2-\ga}. 
\end{align}
Combining \eqref{ll11} and \eqref{ll13} and recalling \eqref{bingo}, we obtain
\begin{align}
  \PrB\bigsqpar{A_0<\taue< 2n }
\le C(\eps) e^{-\eps^2/8\gd}.
\label{ll1q}
\end{align}
Moreover, as \ntoo, by \refL{Lfd},
\begin{align}\label{ll1w}
 \PrB\bigsqpar{ \taue=A_0}
=
\PrB\bigsqpar{|\gD_{A_0}|\ge \eps n\qq}
\to \Pr\bigsqpar{|G_{1-\gd}|\ge\eps}
\end{align}

Hence, combining \eqref{ll1q} and \eqref{ll1w},
\begin{align}
\limsup_\ntoo
  \PrB\Bigsqpar{\sup_{t\ge{1-\gd}}|Z_n(t)|\ge\eps}
&=  \limsup_\ntoo\PrB\bigsqpar{\set{A_0 \le \taue< 2n }}
\nonumber\\&
\le C(\eps) e^{-\eps^2/8\gd} + \Pr\bigsqpar{|G_{1-\gd}|\ge\eps}.
\end{align}
The right-hand side can be made less than $\eta$ by choosing $\gd$
small, which yields \eqref{lsb1}.
\end{proof}

Next, consider the processes only on an interval $[a,1)$, for some fixed
$a\in(0,1)$. 

\begin{lemma}\label{La1}
Let\/ $0<a<1$.
Then $Z_n(t)\dto \Ga(t)$ in $\cD[a,1)$ as \ntoo.
\end{lemma}

\begin{proof}
The space $\cD[a,1)$ of functions, equipped with the Skorohod topology,
is homeomorphic to the space $\Doo$ by a
change of variable. (Any continuous increasing bijection $[a,1)\to[0,\infty)$  
will do; we pick one, for example $t\mapsto(t-a)/(1-t)$.)
It follows that \refL{LT} applies to $\cD[a,1)$ as well, considering only
$T<1$ in \ref{LTa}, \eqref{SB} and \ref{LTb},
and  $t\ge a$  in \ref{LTb}.

The stochastic boundedness on $[a,1)$ is given by \refL{LSB}.
The condition \eqref{lt3} is satisfied, because trivially each jump in
$Z_n(t)$ is $n\qqw$.

It remains to verify \ref{LTb} in \refL{LT}. Let $\ga_n(\gl,\eps,\gd,T)$ be
the smallest number such that \eqref{lt1} holds, i.e., the supremum of the
left-hand side over all $t,u$ and $Z_n(t)$ satisfying the conditions.

Let $n_1:=\floor{2nt}$, $n_2:=\floor{2nu}$ and let $k$ be an integer with
$|k|\le\gl n\qq$.
Then
\begin{align}
&  \PrBx{|Z_n(u)-Z_n(t)|\le\eps\mid Z_n(t)=kn\qqw}
\nonumber\\&\qquad
=  \PrBx{|\gDp{n_2}-\gDp{n_1}|\le\eps n\qq\mid \gDp{n_1}=k}
\nonumber\\&\qquad
=\sum_{|j-k|\le \eps n\qq} \frac{\Prz{\cA_1(j)}\Prz{\cA_2(j)}}{\Prz{\cA_0}},
\label{gran}
\end{align}
where we define the events
\begin{align}
  \cA_1(j)&:=BINGO\Bigpar{\frac{n_1+k}2,\frac{n_1-k}2;\frac{n_2+j}2,\frac{n_2-j}2},
\\
\cA_2(j)&:=BINGO\Bigpar{\frac{n_2+j}2,\frac{n_2-j}2;n,n},
\\
\cA_0&:=BINGO\Bigpar{\frac{n_1+k}2,\frac{n_1-k}2;n,n}.
\end{align}
Let $\phi_t(x)$ and $\kk$ be 
as in \eqref{phit} and \eqref{kk}. Let further, cf.\ \eqref{Ti}--\eqref{yi},
$T_1:=t^{1-2\ga}-1$, $T_2:=u^{1-2\ga}-1$,
$x:=\kk t^{-\ga} k/\sqrt n$ and $y:=\kk u^{-\ga} j/\sqrt n$.
By calculations as in the proof of \refL{Lfd}, see \eqref{b5}, we obtain
\begin{align}\label{tall}
  \frac{\Prz{\cA_1(j)}\Prz{\cA_2(j)}}{\Prz{\cA_0}}
&
\sim \frac{2\kk}{u^\ga\sqrt n}\frac{\phi_{T_1-T_2}(x-y)\phi_{T_2}(y)}{\phi_{T_1}(x)} 
\nonumber\\&
=\frac{2\kk}{u^\ga\sqrt n}
\phi_{{(T_1-T_2)T_2}/{T_1}}\Bigpar{y-\frac{T_2}{T_1}x},
\end{align}
uniformly in all $t, u, k, j$, satisfying the conditions above (cf.\
\refR{Runiform}). 
Using \eqref{gran} and
summing \eqref{tall} 
over all $j$ with $|j-k|\le \eps n\qq$ and $j\equiv n_2\pmod 2$, so 
$y$ takes values with step $2\kk u^{-\ga}/\sqrt n$, we obtain,
with $Z\sim N(0,1)$ a standard normal variable, 
\begin{align}
&  \PrBx{|Z_n(u)-Z_n(t)|\le\eps\mid Z_n(t)=kn\qqw}
\nonumber\\&\qquad
\sim \int_{|u^\ga y -t^{\ga} x|\le\kk \eps}
\phi_{{(T_1-T_2)T_2}/{T_1}}\Bigpar{y-\frac{T_2}{T_1}x}\dd y
\nonumber\\&\qquad
=\PrX{ \Bigabs{\Bigpar{\frac{(T_1-T_2)T_2}{T_1}}\qq Z +\frac{T_2}{T_1}x
  -\frac{t^\ga}{u^\ga}x} 
\le \kk\eps u^{-\ga}},
\label{grann}
\end{align}
uniformly in $t$ and $u$ satisfying the conditions.
Hence, taking complements and recalling the definition of $\ga_n$,
\begin{multline}
\limsup_\ntoo \ga_n(\gl,\eps,\gd,T)
\\
=\sup
\PrX{ \Bigabs{\Bigpar{\frac{(T_1-T_2)T_2}{T_1}}\qq Z +\frac{T_2}{T_1}x
  -\frac{t^\ga}{u^\ga}x} 
> \kk\eps u^{-\ga}},
\label{grannas}
\end{multline}
taking the supremum over $t,u,x$ with
$a\le t\le u\le T$, $t+\gd\le u\le t+2\gd$ and $|x|\le \kk t^{-\ga}\gl\le
\kk a^{-\ga}\gl$.
For fixed $\eps,\gl>0$ and $T<1$, if $\gd\to0$, then $t/u\to 1$, $T_2/T_1\to
1$ and $T_1-T_2\to 0$, uniformly for all $t$ and $u$ satisfying these
conditions.
It follows from \eqref{grannas} that
$\limsup_\ntoo \ga_n(\gl,\eps,\gd,T)\to 0$ as $\gd\to0$, which verifies
\ref{LTb} in \refL{LT}.

We have verified the conditions in \refL{LT}, and the lemma thus shows that
the sequence $Z_n(t)$ is tight in $\cD[a,1)$.
Combined with the finite-dimensional convergence in \refL{Lfd}, this shows 
convergence to $G_\ga(t)$ in $\cD[a,1)$.
\end{proof}

\begin{proof}[Proof of \refT{Toi}]
\refL{La1} shows $Z_n(t)\dto \Ga(t)$ in $\cD[a,1)$ for every $a\in(0,1)$.
This can equivalently be expressed as convergence in $\cD(0,1)$, considering
the open interval $(0,1)$, and this can be improved to convergence in $D[0,1]$
using Lemmas \ref{LSB0} and \ref{LSB1}, see e.g.\ \cite[Proposition 2.4]{SJI}.

A direct proof can be made as follows.
Let $\eps,\eta>0$ be given.
Find $\gd>0$ such that \eqref{lsb0} and \eqref{lsb1} hold, and furthermore
\begin{align}
\PrX{\sup_{0\le t\le \gd}|\Ga(t)| >\eps}\le\eta,
&&&
\PrX{\sup_{1-\gd\le t\le 1}|\Ga(t)| >\eps}\le\eta.
\end{align}
By \refL{La1} with $a=\gd$, $Z_n(t)\dto \Ga(t)$ in $\cD[\gd,1)$, and thus in
$\cD[\gd,1-\gd]$.
By the Skorohod coupling theorem
\cite[Theorem 4.30]{Kallenberg},
we may assume that $Z_n(t)\to \Ga(t)$ uniformly on  $[\gd,1-\gd]$.
Then, now writing $\Pr$ instead of $\PrB$,
\begin{align}
&\limsup_\ntoo \PrX{\sup_{0\le t\le 1}|Z_n(t)-\Ga(t)|>2\eps}
\nonumber\\&\qquad
\le 4\eta + 
 \limsup_\ntoo \PrX{\sup_{\gd\le t\le 1-\gd}|Z_n(t)-\Ga(t)|>2\eps}
= 4 \eta.
\end{align}
Since $\eta$ and $\eps$ are arbitrary, this shows
$Z_n(t)\to\Ga(t)$ in $\cD\oi$ in probability, and thus in distribution.
\end{proof}

\subsection{The initial part of the process}\label{SS:on}

\begin{proof}[Proof of \refT{Too}]
The proof is very similar to the proof of \refT{Toi}, and we only point out
the differences. 
For convenience, we change the notation by replacing $n$ by $N$ and $m_n$ by
$2n$. Note that the normalization then is by $(2n)\qqw$, differing by a
factor $2\qqw$ from the one in \refT{Toi}.

Finite-dimensional convergence is proved as in \refL{Lfd}.
The main difference is that the probability of the last event $\cA_{m+1}$ is
estimated using \refC{Cbingo0} instead of \refC{Cbingo2}, and that we define
$T_i:=t_i^{1-2\ga}$, where now $t_i\in(0,\infty)$. Then the same calculations
as before show that \eqref{krhmf} holds, which yields finite-dimensional
convergence in \eqref{too}.

Stochastic boundedness on any interval $[a,b]$ with $0<a<b<\infty$ follows
as in \refL{LSB}, again using \refC{Cbingo0}.

The convergence \eqref{too} in the space $\cD[a,\infty)$ for any $a>0$ now
follows as in \refL{La1}, again using \refL{LT}.
This is equivalent to convergence in $\cD(0,\infty)$.

Finally, the analogue of \refL{LSB0} holds, by the same proof with trivial
modifications, which yields convergence also in $\cD\ooo$.
\end{proof}

\begin{proof}[Proof of \refT{thm:oulimit}]
  Apply \refT{Too} with $m_n:=e^{t_n}$. Then \eqref{too} holds in $\Doo$,
  and thus in $\cD(0,\infty)$. The change of variables $t=e^s$ yields
  \begin{equation}\label{ptou2}
e^{-t_n/2} \gDp{\floor{e^{s+t_n}}}
\dto (2\ga-1)\qqw e^{(1-\ga)s} B\bigpar{e^{(2\ga-1)s}},
  \end{equation}
in $\cD(-\infty,\infty)$.
Multiplying \eqref{ptou2} by the continuous function $e^{-s/2}$ yields 
  \begin{equation}\label{ptou3}
e^{-(s+t_n)/2} \gDp{\floor{e^{s+t_n}}}
\dto Z(s):=(2\ga-1)\qqw e^{(\frac12-\ga)s} B\bigpar{e^{(2\ga-1)s}},
  \end{equation}
in $\cD(-\infty,\infty)$, which is \eqref{tou1}.
The covariance \eqref{tou2} follows from the definition of $Z(s)$ in
\eqref{ptou3}. 
\end{proof}

\appendix

\section{Quadratic functionals of Gaussian variables}\label{App}

A \emph{Gaussian Hilbert space} is a closed subspace $H$
of $L^2\OFP$, for some probability space $\OFP$,
such that every element $f$ of $H$ is a  random variable with a
centered Gaussian distribution $N(0,\gss)$ (where $\gss=\norm{f}_2^2$).
(In this appendix, we consider real vector spaces and real-valued
functions; thus $L^2=L^2_\bbR$ is the space of real-valued square
integrable functions.)
We review here, for the readers' and our own convenience, 
some basic and more or less well-known facts; 
further details can be found \eg{} in 
\cite{SJIII}.
In our application above, $H$ is the closed linear span of 
$\set{B_t:0\le  t\le n}$, 
as usual defined on some anonymous probability space $\OFP$,
but we state the results generally.

If $\xi,\eta\in H$, define their \emph{Wick product} by
$\wick{\xi\eta}:=\xi\eta-\E(\xi\eta)$, i.e., the centered product.
Let $\HH$ be the closed linear span of all Wick products $\wick{\xi\eta}$,
$\xi,\eta\in H$. 
Hence, if $X=Q(\xi_1,\dots,\xi_N)$ is a quadratic form in
Gaussian random variables $\xi_1,\dots,\xi_N\in H$, then $X-\E X\in \HH$,
and conversely, every element of $\HH$ is a limit of such centered quadratic
forms, and can be written as a quadratic form in (in general) infinitely
many variables. Moreover, this form can be diagonalized by a suitable choice
of orthonormal basis in $H$, leading to the following representation
theorem. 
(Note that every $X\in\HH$ has $\E X=0$ as a consequence of the definition.)

\begin{theorem}[{\cite[Theorem 6.1]{SJIII}}]\label{T61}
If $X\in\HH$, then there exists a finite or infinite sequence
$(\gl_j)_{j=1}^N$, $0\le N\le \infty$, of non-zero real numbers
such that $\sum_j\gl_j^2<\infty$ and
\begin{equation}\label{t61}
  X\eqd \sum_{j=1}^N \gl_j\bigpar{\xi_j^2-1},
\end{equation}
where $\xi_j$ are \iid{} $N(0,1)$ random variables.
The numbers $\gl_j$ are the non-zero eigenvalues (counted with
multiplicities) of the compact symmetric bilinear form
\begin{equation}\label{Bform}
  \BB_X(\xi,\eta):=\tfrac12\E(X\xi\eta),
\qquad \xi,\eta\in H,
\end{equation}
or, equivalently, of the corresponding compact symmetric operator $\xB_X$ on
$H$ defined by $\innprod{\xB_X(\xi),\eta}=\BB_X(\xi,\eta)=\tfrac12\E(X\xi\eta)$,
$ \xi,\eta\in H$.
\end{theorem}

In particular, \eqref{t61} yields the moment generating function,
for all real $t$ such that $2\gl_j t<1$ for every $j$,
\begin{equation}\label{mgf1}
  \E e^{tX}=\prod_j \bigpar{1-2\gl_j t}\qqw e^{-\gl_j t}.
\end{equation}

In our application, we deal with non-centered quadratic functionals, and
then the following version is more directly applicable.
(Cf.\ the more general but less specific \cite[Theorem 6.2]{SJIII}.
We do not know a reference for the precise statements in
\refT{T+}, so we give a complete proof.)

\begin{theorem}\label{T+}
Suppose that 
\begin{romenumerate}
\item \label{T+a}
$X$ is a random variable such that $X-\E X\in\HH$;
\item \label{T+b}
$X\ge 0$ a.s., and $\P(X<\eps)>0$ for every $\eps>0$
(i.e., the lower bound of the support of $X$ is $0$);
\item \label{T+c}
the bilinear form $\BB=\BB_{X-\E X}$
is positive, i.e.,
\begin{equation}
\BB_{X-\E X}(\xi,\xi)
= \tfrac12\E \bigpar{(X-\E X)\xi^2}
\ge0  
\end{equation}
for every $\xi\in H$.
\end{romenumerate}
Then 
\begin{equation}\label{tab}
  X\eqd \sum_{j=1}^N \gl_j\xi_j^2,
\end{equation}
where $\xi_j$ are \iid{} $N(0,1)$ random variables and
the coefficients $\gl_j>0$ are the non-zero eigenvalues (counted with
multiplicities) of $\BB$.
Furthermore,\begin{equation}\label{trace}
 \E X= \sum_{j=1}^N \gl_j<\infty
\end{equation}
and, for  $-\infty<t<(2 \max_j\gl_j)\qw$,
\begin{equation}\label{mgf+}
  \E e^{tX}=\prod_j \bigpar{1-2\gl_j t}\qqw .
\end{equation}
\end{theorem}

\begin{proof}
  \refT{T61} yields the representation
  \begin{equation}\label{a7}
X-\E X = \sum_j \gl_j\bigpar{\xi_j^2-1},
  \end{equation}
where $\gl_j>0$ since the positive form $\BB$ has only non-negative
eigenvalues. Hence, \eqref{mgf1} applies for any $t<0$, and thus, replacing
$t$ by $-t$, for every $t>0$,
\begin{equation}\label{a8}
  \frac1t\ln\E e^{-t(X-\E X)}
=\sum_j \frac{t\gl_j-\frac12\ln(1+2\gl_j t)}{t}.
\end{equation}
Now let $t\to\infty$. By \ref{T+b}, $\E e^{-tX}\le1$, but
$\liminf_{\ttoo}\frac1t\ln \E e^{-tX}\ge -\eps$ for every $\eps>0$, and thus 
$\frac1t\ln \E e^{-tX}\to0$ and $\frac1t\ln \E e^{-t(X-\E X)}\to \E X$.
In the sum on the \rhs{} of \eqref{a8}, each term is positive, and increases
to $\gl_j$ as $\ttoo$ (because $\ln$ is concave);
hence the sum tends to $\sum_j\gl_j$ by monotone convergence.
Consequently, $\E X=\sum_j\gl_j$, and since $\E X<\infty$, \eqref{trace} holds.

The representation \eqref{tab} follows from \eqref{a7} and \eqref{trace},
and \eqref{mgf+} is an immediate consequence.
\end{proof}

\begin{remark}\label{Rcompact}
The operator $\xB$  in \refT{T61} is a Hilbert--Schmidt operator, since
$\sum_j\gl_j^2<\infty$. Similarly, in \refT{T+}, $\xB$ is a trace class
operator, with trace and trace norm $\sum_j\gl_j=\E X$.
\end{remark}

If $Y\in H$, then $X:=Y^2$ satisfies \ref{T+a} and \ref{T+b} in \refT{T+}.
Furthermore, for $\xi,\eta\in H$, since $X-\E X= Y^2-\E Y^2=\wick{Y^2}$,
by \cite[Theorem 3.9]{SJIII},
\begin{equation}\label{a10}
  \BB_{X-\E X}(\xi,\eta)=\tfrac12\E \bigpar{\wick{Y^2}\xi\eta}
=\tfrac12\E \bigpar{\wick{Y^2}\wick{\xi\eta}} 
=\E (Y\xi)\E (Y\eta).
\end{equation}
Taking $\eta=\xi$, we see that  $\BB$ is a positive form.
Hence the conditions \ref{T+a}--\ref{T+c} hold for $X=Y^2$, and it follows
that they hold also for any finite sum of squares $\sum_i Y_i^2$ with
$Y_i\in H$,  
for example $\tL_n$ in \eqref{tL}. Moreover, we can take limits, and
conclude that the conditions also hold for, for example, the integral
$\hL_n$ in \eqref{hL}. (Note that \ref{T+b} is obvious for $\hL_n$.)
Hence, \refT{T+} applies to $\hL_n$.

\subsection{Stochastic integration}

In order to find the eigenvalues $\gl_i$ in Theorem  \ref{T+}
in our application to $\hL_n$,
it is convenient to transfer from the Gaussian Hilbert space $H$ to the
function space $L^2\ono$ by means of stochastic integrals. 

The stochastic integral $\inton f(t)\dd B_t$ can be defined for every
(deterministic) function $f\in L^2\ono$ as follows.
(This is a simple form of stochastic integrals; we have no need for the
general theory of random integrands here.)

First, $\inton \etta_{(0,a)}(t)\dd B_t=B_a$ for every $a\in[0,n]$.
This and linearity defines $\inton f(t)\dd B_t$ for every step function $f$
(in the obvious, naive way). A simple calculation shows that 
\begin{equation}
  \E\Bigpar{\inton f(t)\dd B_t}^2 = \inton f(t)^2\dd t.
\end{equation}
Hence, 
the mapping
$I:f\mapsto \int f\dd B_t$ is an isometry from the subspace of step functions
in $L^2\ono$ to the linear span of the random variables $B_t$, $t\in\on$.
We let $H$ be the closure of the latter space, regarded as a subspace of
$L^2\OFP$; then $I$ extends by continuity to an isometry $I: L^2\ono\to H$.
We may write $I(f)=\inton f(t)\dd B_t$.
This isometry enables us to regard the bilinear form $\BB$ and operator $\xB$
as defined on $L^2\ono$.

If $Y=I(f)$, $\xi=I(g)$ and $\eta=I(h)$, for some $f,g,h\in L^2\ono$, 
and $X=Y^2$, then \eqref{a10} yields
\begin{equation}
  \BB(\xi,\eta)=\E(Y\xi)\E(Y\eta)=\innprod{f,g}\innprod{f,h}
=\inton\inton f(x)f(y)g(x)h(y)\dd x\dd y,
\end{equation}
which shows that $\xB$, regarded as an operator on $L^2\ono$, 
is the integral operator with kernel $f(x)f(y)$.

\begin{example}\label{EhL}
  For any $t\in\on$, $B_t=I(\etta_{(0,t)})$, and thus
$X=B_t^2$ corresponds to the integral operator $\xB$ with kernel
$\etta_{(0,t)}(x)\etta_{(0,t)}(y)$.
It follows easily 
that $\hL_n=\int_1^n \bigpar{B_t^2/t^2} \dd t$ corresponds to
the integral operator with kernel
\begin{equation}\label{kernel}
K_n(x,y):=\int_1^n \frac{1}{t^2}\etta_{(0,t)}(x)\etta_{(0,t)}(y) \dd t
=\int_{1\lor x\lor y}^n\frac{ \dd t}{t^2}
=\frac{1}{1\lor x\lor y}-\frac{1}{n}.
\end{equation}
\end{example}

We summarize as follows.
\begin{lemma}\label{LhL}
\refT{T+} applies to $X=\hL_n$, with
$\gl_j$ the non-zero eigenvalues of the symmetric integral operator on
$L^2\ono$ with kernel $K_n$ given by \eqref{kernel}.
\end{lemma}

\end{document}